\UseRawInputEncoding

\documentclass[a4paper,12pt]{article}

\usepackage{amsmath,amsfonts,amssymb,amsthm,amscd}
\usepackage {amsfonts, amssymb, amscd,amsthm, amsmath, eucal}

\usepackage{amsbsy}
\usepackage[all]{xy}

\theoremstyle{plain} {
  \newtheorem{thm}{Theorem}[section]
  \newtheorem{defn}[thm]{Definition}
  \newtheorem{cor}[thm]{Corollary}
  \newtheorem{lem}[thm]{Lemma}
  \newtheorem{prop}[thm]{Proposition}
  \theoremstyle{definition}
  \newtheorem{rem}[thm]{Remark}
    
    \newtheorem{exam}[thm]{Example}
  \theoremstyle{plain}
  
  \newtheorem{notation}[thm]{Notation}
  \newtheorem{agr}[thm]{Agreement}
  \newtheorem{sssection}[thm]{}
}

{

}
\renewcommand{\subsubsection}{\sssection\rm}

\newcommand{\Aff}{\mathbf {A}}
\newcommand{\Pro}{\mathbf {P}}

\newcommand \hra {\hookrightarrow }

\newcommand \Tr{\text{Tr}}

\renewcommand \phi\varphi

\begin{document}

\title{Moving lemmas in mixed characteristic and applications
}

\author{Ivan Panin\footnote{February the 1-st of 2022}
}

\date{Steklov Mathematical Institute at St.-Petersburg}

\maketitle

\begin{abstract}
The present paper contains new geometric theorems proven in mixed characteristic case.
We derive a bunch of cohomological consequences using these geometric theorems.
Among them an isotropy result for quadratic spaces, a purity result for quadratic spaces,
Grothendieck--Serre conjecture for groups $SL_{1,D}$, where $D$ is an Azumaya algebra.
The Gersten conjecture for the functor $K_2$ is proved.
Bloch-Ogus type result is obtained as well. Suslin's exact sequence
is derived and its application to a finiteness result is given.
A version of the Roitman theorem is proved.
A very general result concerning
the Cousin complex is obtained for any cohomology theory
in the sense of Panin--Smirnov.
\end{abstract}

\section{Introduction}
Most of the result of this preprint are new. However some of them are well-known
in the literature:
\cite{BW}, \cite{B-F/F/P}, \cite{BP}, \cite{C-T/S}, \cite{C-T/S 87}, \cite{CHK}, \cite{C}, \cite{F}, \cite{Dru}, \cite{D-K-O}
\cite{G}, \cite{Gi1}, \cite{Gi2}, \cite{GiP}, \cite{Fe1}, \cite{Fe2}, \cite{J}, \cite{GL}, \cite{N},
\cite{Oj1}, \cite{Oj2}, \cite{Pa1}, \cite{Pa2}, \cite{SS}.
There is a good hope that developing futher methods of the present preprint one can get more applications.
We decided to postpone these to a future. Also we postpone to another preprint
proofs of three main geometric theorems formulated in
Section \ref{thms:geometric}.
Point out that \cite[Proposition 2.2.1]{C} was the starting inspiring point of the departure.

\section{Agreements}\label{agreements}
Through the paper \\
$A$ is a d.v.r., $m_A\subseteq A$ is its maximal ideal;\\
$\pi \in m_A$ is a generator of the maximal ideal; \\
$K$ is the field of fractions of $A$;
$V=Spec(A)$ and $v\in V$ is its closed point;\\
$k(v)$ is the residue field $A/m_A$;\\
$p>0$ is the characteristic of the field $k(v)$; \\
it is supposed in this notes that the field $K$ has characteritic zero;\\
$d\geq 1$ is an integer;\\
$X$ is an $A$-smooth $A$-scheme irreducible and equipped with a finite surjective $A$-morphism
$X\to \Aff^d_A$;
{\bf Particularly, all open subschemes of $X$ are regular and all its local rings are regular.}.
For an $A$-scheme $Y$ we write $Y_v$ for the closed fibre of the structure morphism $Y\to V$.\\
If $x_1,x_2,\dots,x_n$ are closed points in the scheme $X$, then write \\
$\mathcal O$ for the semi-local ring $\mathcal O_{X,\{x_1,x_2,\dots,x_n\}}$, $\mathcal K$ for the fraction field of $\mathcal O$, \\
$U$ for $Spec(\mathcal O)$, $\eta$ for $Spec(\mathcal K)$,\\
$X'\subseteq X$ for an affine Zariski open which {\bf contains all generic points of the scheme $X_v$},
but does not contain any irreducible component of the scheme $X_v$.\\\\
Examples. Let $\bar X$ be an $A$-smooth projective irreducible $A$-scheme.\\
Let $X_{\infty}\subseteq X$ be a divizor which does not contain any component
of the closed fibre of $X$ over $V$. Then $X=X-X_{\infty}$
is an $A$-scheme of our interest in this paper.\\\\
More generally: let $\bar X$ be an $A$-projective irreducible $A$-scheme which is normal.
Let $X_{\infty}\subseteq X$ be an effective Cartie divizor such that \\
a) the $A$-scheme $X=X-X_{\infty}$ is $A$-smooth; \\
b) the divisor $X_{\infty}$ does not contain any component
of the closed fibre of $X$ over $V$.\\
Then $X=X-X_{\infty}$
is an $A$-scheme of our interest in this paper.\\\\
all schemes are supposed to be separated \\\\
$Sm/V$ is the category of $V$-smooth schemes of finite type over $V$ and $V$-morphisms;\\\\
$Sm'/V$ is the category of essentially smooth $V$-schemes. Following [10],
by an essentially smooth $V$-scheme we mean a Noetherian $V$-scheme $Y$
which is the inverse limit of a left filtering system $(Y_i)_{i\in I}$ with each
transition morphism $Y_i \to Y_j$ being an \'{e}tale affine $V$-morphism between
smooth $V$-schemes of finite type;\\\\
$Sm'Op/V$ is a category whose objects are pairs $(X,W)$,
where $X\in Sm'/V$
and $W \subseteq X$ is an open subscheme. A morphism
$f: (X,W)\to (X',W')$ in $Sm'Op/V$ is just a $V$-morphism
$f: X\to X'$ such that $f(W)\subseteq W'$.\\\\
$Sch/V$ is the category of $V$-schemes of finite type and $V$-morphisms;\\\\
$Sch'/V$ is the category of $V$-schemes essentially of finite type. Recall that
by a $V$-scheme essentially of finite type we mean a Noetherian $V$-scheme $S$
which is the inverse limit of a left filtering system $(S_j)_{j\in J}$ with each
transition morphism $S_j \to S_{j'}$ being an affine $V$-morphism between
$V$-schemes of finite type;\\\\
if $l$ is a prime and $B$ is an abelian group, then $_{\{l\}}B$ denotes the subgroup of $l$-primary elements in $B$.




\section{Geometric theorems}\label{thms:geometric}

The main aim of the preprint is to state the following three geometric results and to get
plenty of their applications. Let $A$ be a d.v.r. and $X$ be an $A$-scheme as in the section
\ref{agreements}.

\begin{thm}[geometric presentation]\label{geom_pres_mixed_char}
Let $x_1,x_2,\dots,x_n$ be closed points in the scheme $X$. Let $X'$ be a neighborhood of points
$x_1,x_2,\dots,x_n$ which {\bf contains all generic points of the scheme $X_v$},
but does not contain any irreducible component of the scheme $X_v$.
Then one can find inside $X'$ an affine neighborhood $X^{\circ}$ of points
$x_1,x_2,\dots,x_n$, an open affine subscheme $S\subseteq \Pro^{d-1}_A$ and
{\bf a smooth $A$-morphism}
$$q: X^{\circ}\to S$$
equipped with a finite surjective $S$-morphism $\Pi: X^{\circ}\to \Aff^1_S$.

Let $Z$ be a closed subset in $X'$ of codimension at least 2 in $X'$ such that each its irreducible component
contains at least one of the point $x_i$'s. Then one can choose the scheme $X^{\circ}$ and
the morphism
$q$ such that additionally $Z^{\circ}/S$ is finite, where $Z^{\circ}=Z\cap X^{\circ}$.

Let $Y$ be a closed subset in $X'$ of codimension 1 such that
$Y$ does not contain any irreducible component of the scheme $X'_v$.
Then one can choose the scheme $X^{\circ}$ and
the morphism
$q$ such that additionally
$Y^{\circ}/S$ is quasi-finite and $Y^{\circ}_v$ is not-empty, where $Y^{\circ}=Y\cap X^{\circ}$.
\end{thm}
Here are two consequences of this Theorem.

\begin{cor}\label{Homotopy_w_Traces}
Let $X$, $x_1,x_2,\dots,x_n \in X$, $X'\subseteq X$ and $Y\subset X'$ be as in Theorem \ref{geom_pres_mixed_char}.
Then there is
an affine Zariski open neighborhood $U'$ of points
$x_1,x_2,\dots,x_n \in X'$ and
a diagram of $A$-schemes and $A$-morphisms of the form
$$\Aff^1_{U'} \xleftarrow{\tau} \mathcal X \xrightarrow{p_X} X'$$
with an $A$-smooth irreducible scheme $\mathcal X$ and a finite surjective morphism $\tau$
such that the canonical sheaf \ $\omega_{\mathcal X/V}$ is isomorphic to the structure sheaf $\mathcal O_{\mathcal X}$.
Moreover, if we write $p: \mathcal X\to U'$ for the composite map $pr_{U'}\circ \tau$ and $\mathcal Y$ for $p^{-1}_X(Y)$, then these data
enjoy the following properties:\\
0) the closed subset $\mathcal Y$ of $\mathcal X$ is quasi-finite over $U'$;\\
1) there is a section $\Delta: U'\to \mathcal X$ of the morphism $p$ such that $\tau\circ \Delta=i_0$ and $p_X\circ \Delta=can$,
where $i_0$ is the zero section of $\Aff^1_{U'}$ and $can: U'\hookrightarrow X'$ is the inclusion;\\
2) for $\mathcal D_1:=\tau^{-1}(\{1\}\times U')$ one has $\mathcal D_1\cap \mathcal Y=\emptyset$; \\
3) for $\mathcal D_0:=\tau^{-1}(\{0\}\times U')$ one has $\mathcal D_0=\Delta(U')\sqcup \mathcal D'_0$ and $\mathcal D'_0\cap \mathcal Y=\emptyset$;\\
\end{cor}

\begin{cor}\label{Homotopy_w_Traces_2}
Let $X$, $x_1,x_2,\dots,x_n \in X$, $X'\subseteq X$ and $Y\subset X'$ be as in Theorem \ref{geom_pres_mixed_char}.
Suppose the field $k(v)$ is infinite.
Then there is
an affine Zariski open neighborhood $U'$ of points
$x_1,x_2,\dots,x_n \in X'$ and
a diagram of $A$-schemes and $A$-morphisms of the form
$$\Aff^1_{U'} \xleftarrow{\tau} \mathcal X \xrightarrow{p_X} X'$$
with an $A$-smooth irreducible scheme $\mathcal X$ and a finite surjective morphism $\tau$
such that the canonical sheaf \ $\omega_{\mathcal X/V}$ is isomorphic to the structure sheaf $\mathcal O_{\mathcal X}$.
Moreover, if we write $p: \mathcal X\to U'$ for the composite map $pr_{U'}\circ \tau$ and $\mathcal Y$ for $p^{-1}_X(Y)$, then these data
enjoy the following properties:\\
0) the closed subset $\mathcal Y$ of $\mathcal X$ is quasi-finite over $U'$;\\
1) there is a section $\Delta: U'\to \mathcal X$ of the morphism $p$ such that $\tau\circ \Delta=i_0$ and $p_X\circ \Delta=can$,
where $i_0$ is the zero section of $\Aff^1_{U'}$ and $can: U'\hookrightarrow X'$ is the inclusion;\\
2) the morphism $\tau$ is \'{e}tale as over $\{0\}\times U'$, so over $\{1\}\times U'$;\\
3) for $\mathcal D_1:=\tau^{-1}(\{1\}\times U')$ one has $\mathcal D_1\cap \mathcal Y=\emptyset$; \\
4) for $\mathcal D_0:=\tau^{-1}(\{0\}\times U')$ one has $\mathcal D_0=\Delta(U')\sqcup \mathcal D'_0$ and $\mathcal D'_0\cap \mathcal Y=\emptyset$;\\
5) let $(P',\varphi')$ be a quadratic space over $X'$ and $(P,\varphi)=can^*(P',\varphi')$, then one can construct the above diagram
such that the properties (1) to (4) does hold and the quadratic spaces $p^*_{U'}(P,\varphi)$ and $p^*_X(P',\varphi')$ are isomorphic.
\end{cor}
One more consequence of the Theorem \ref{geom_pres_mixed_char} is the following moving lemma.

\begin{thm}[A moving lemma]
\label{An extended_moving_lemma}
Let $X$, $x_1,x_2,\dots,x_n \in X$, $X'\subseteq X$ and $Z\subset X'$ be as in Theorem \ref{geom_pres_mixed_char}.
Let $c \geq 2$ be an integer and suppose
$Z$ has pure codimension $c$ in $X'$.
Then there are an affine Zariski open neighborhood $U'$ of points
$x_1,x_2,\dots,x_n \in X'$
and a closed subset $Z^{new}$ in $X'$ containing $Z$ of pure codimension $c-1$
and a morphism of pointed Nisnevich sheaves
$$\Phi_t: \Aff^1 \times U'/(U'-Z^{new}) \to X'/(X'-Z)$$
in the category $Shv_{nis}(Sm/V)$
such that for $\Phi_0=\Phi\circ i_0$ and $\Phi_1=\Phi\circ i_1$ one has:\\
(1) $\Phi_0: U'/(U'-Z^{new})\to X'/(X'-Z)$ is the composite morphism \\
$U'/(U'-Z^{new})\xrightarrow{can} X'/(X'-Z^{new})\xrightarrow{p} X'/(X'-Z)$;\\
(2)
$\Phi_1: U'/(U'-Z^{new})\to X'/(X'-Z)$ takes everything
to
the point $*$ in $X'/(X'-Z)$.
\end{thm}



\section{General results on Cousin complexes}
Let $A$, $p>0$, $d\geq 1$, $X$, $x_1,x_2,\dots,x_n\in X$, $\mathcal O$ and $U$  be as in Section \ref{agreements}.

\begin{defn}\label{presh:contenious}
One says that a presheaf $G$ on $Sm'/V$ is continuous
if for each $Y\in Sm'/V$, each left filtering system $(Y_i)_{i\in I}$
as in section \ref{agreements}
and each $V$-scheme isomorphism $Y\to \lim_{i\in I} Y_i$
the map
$$\text{colim}_{i\in I}G(Y_i)\to G(Y)$$ is an isomorphism.
\end{defn}

\begin{defn}\label{def:Z_graded}
Let $E: (Sm'Op/V)^{op} \to Ab$ together with
$$(X,W)\mapsto [\partial_{X,W}: E(W)\to E(X,W)]$$
be a cohomology theory on the category $Sm'Op/V$
in the sense of \cite[Definition 2.1]{PW} (inspired by \cite[Definition 2.1]{PSm}).
One says that $E$ is $\mathbb Z$-graded, if for each pair $(X,W)\in Sm'Op/V$
the group
$E(X,W)$ is a graded abelian group $\oplus E^n(X,W)$ and the differential
$\partial_{X,W}: E(W)\to E(X)$
is a graded abelian group homomorphism of degree $+1$.

It is convenient to suppose in this paper that all cohomology theories are
$\mathbb Z$-graded.
\end{defn}

\begin{defn}\label{contenious}
One says that a $\mathbb Z$-graded cohomology theory $(E,\partial)$ is continuous if
for each integer $n$ the presheaf $E^n$ on $Sm'/V$ is continuous.
\end{defn}

\begin{agr}\label{theory:cont}
Over this preprint each $\mathbb Z$-graded cohomology theory $(E,\partial)$
with the character "E" on the first place is
supposed to be continuous.
\end{agr}

\begin{rem}
It turns out that all specific $\mathbb Z$-graded cohomology theory regarded
in this preprint are continuous.
\end{rem}
Let $z\in U$ be a point. We write in this preprint
$E^m_z(U)$ for $E^m_z(Spec~\mathcal O_{X,z})=E^m_z(Spec~\mathcal O_{U,z})$.
For each integers $m$ and $c$ with $c\geq 0$ and each codimention $c$ point $z\in U$ abelian groups
$E^m_{\geq c}(U)$, $E^m_{(c)}(U)$ and maps
$\partial: E^m_{(c)}(U)\to E^{m+1}_{\geq c+1}(U)$ are defined
in \cite[Section 9]{Pan0}.
By the abuse of notation the same symbol $\partial$ is used there to denote the composite map
$E^m_{(c)}(U)\xrightarrow{\partial} E^{m+1}_{\geq c+1}(U)\to E^{m+1}_{(c+1)}(U)$.
By the construction
$0=\partial \circ \partial: E^m_{(c)}(U)\to E^{m+2}_{(c+2)}(U)$. Combining all these together
the complex of the form
\begin{equation}\label{Cousin_prelim}
0\to E^m(U) \xrightarrow{\eta^*} E^m_{(0)}(U) \xrightarrow{\partial} E^{m+1}_{(1)}(U) \xrightarrow{\partial} ... \xrightarrow{\partial}
E^{m+d+1}_{(d+1)}(U) \to 0
\end{equation}
is defined in \cite[Section 9, (8)]{Pan0}.
Using the excision property of $(E,\partial)$ the equalities
$E^m_{(c)}(U)=\oplus_{z\in U^{(c)}} E^{m}_{z}(U)$ are explained in \cite[Section 9, (8)]{Pan0}.
In this way the complex
\begin{equation}\label{thm:Cousin_Exact}
0\to E^m(U) \xrightarrow{\eta^*} E^m(\eta) \xrightarrow{\partial} \oplus_{y\in U^{(1)}} E^{m+1}_{y}(U) \xrightarrow{\partial} ... \xrightarrow{\partial}
\oplus_{x\in U^{(d+1)}}E^{m+d+1}_{x}(U) \to 0
\end{equation}
is constructed in \cite[Section 9, (8)]{Pan0}.

\begin{notation}\label{General_Cousin}
Write $Cous(E,U;m)$ for the complex \eqref{thm:Cousin_Exact}.
\end{notation}

\begin{thm}\label{General_Purity}
Let $m$ and $c$ be an integers and $c\geq 0$. Then \\
(a) the complex \eqref{thm:Cousin_Exact} is exact except possibly the terms
$E^m(U)$ and $\oplus_{y\in U^{(1)}} E^{m+1}_{y}(U)$; \\
(b) particularly, the complex \eqref{thm:Cousin_Exact}
is exact at the term $E^m(\eta)$;\\
(c) for each integer $m$ and $c\geq 1$ the map $E^{m}_{\geq c+1}(U)\to E^m_{\geq c}(U)$ vanishes;\\
(d) for each integer $m$ and $c\geq 1$ the sequence
$0\to E^m_{\geq c}(U)\to E^m_{(c)}(U)\xrightarrow{\partial} E^{m+1}_{\geq c+1}(U)\to 0$
is short exact.
\end{thm}

\begin{rem}\label{General_Purity_2}
The item (b) states the following.
Let $\alpha\in E^m(\eta)$ be an element such that for each codimension one point $y\in U$
it can be lifted to an element $\alpha_y \in E^m(\mathcal O_{X,y})$. Then $\alpha$ can be lifted to an element
$E^m(\mathcal O)$. So,
{\it purity holds for the presheaf $E^m$ on $Sm/V$ and the ring $\mathcal O$.
}
\end{rem}

\begin{thm}\label{cousin_1}
If the maps $\eta^*: E^i(U)\to E^i(\eta)$ are injective for $i=m$ and $i=m+1$, then the Cousin complex
$Cous(E,U;m)$ is exact.
\end{thm}

Let $z\in X$ be a codimension $c(z)$ point. We write in this preprint
$U_z$ for $Spec~\mathcal O_{X,z}$. If $y$ is a point in $U_z$ we write
$E^m_y(U_z)$ for $E^m_y(Spec~\mathcal O_{X,y})=E^m_y(Spec~\mathcal O_{U,y})$.
Following the procedure of constructing the complex
\eqref{thm:Cousin_Exact}
one can construct a complex
of the form
\begin{equation}\label{thm:Cousin_Exact_z}
0\to E^m(U_z) \xrightarrow{\eta^*} E^m(\eta) \xrightarrow{\partial} \oplus_{y\in U^{(1)}_z} E^{m+1}_{y}(U_z) \xrightarrow{\partial} ... \xrightarrow{\partial}
E^{m+c(z)}_{z}(U_z) \to 0
\end{equation}

\begin{notation}\label{General_Cousin_z}
Write $Cous(E,U;m)$ for the complex \eqref{thm:Cousin_Exact_z}.
\end{notation}

\begin{thm}\label{General_Purity_z}
Let $m$, $c$ be an integers, $c\geq 0$,
$z\in X$ be a point, $\text{codim}_X z=c(z)$.
Then \\
(a) the complex \eqref{thm:Cousin_Exact} is exact except possibly the terms
$E^m(U)$ and $\oplus_{y\in U^{(1)}} E^{m+1}_{y}(U)$; \\
(b) particularly, the complex \eqref{thm:Cousin_Exact}
is exact at the term $E^m(\eta)$;\\
(c) for each integer $m$ and $c\geq 1$ the map $E^{m}_{\geq c+1}(U)\to E^m_{\geq c}(U)$ vanishes;\\
(d) for each integer $m$ and $c\geq 1$ the sequence
$0\to E^m_{\geq c}(U)\to E^m_{(c)}(U)\xrightarrow{\partial} E^{m+1}_{\geq c+1}(U)\to 0$
is short exact.
\end{thm}

\begin{rem}\label{General_Purity_2_z}
The item (b) states the following.
Let $\alpha\in E^m(\eta)$ be an element such that for each codimension one point $y\in U_z$
it can be lifted to an element $\alpha_y \in E^m(\mathcal O_{X,y})$. Then $\alpha$ can be lifted to an element
$E^m(\mathcal O_{X,z})$. So,
{\it purity holds for the presheaf $E^m$ on $Sm/V$ and the ring $\mathcal O_{X,z}$.
}
\end{rem}

\begin{thm}\label{cousin_1_z}
If the maps $\eta^*: E^i(U_z)\to E^i(\eta)$ are injective for $i=m$ and $i=m+1$, then the Cousin complex
$Cous(E,U_z;m)$ is exact.

Particularly, if for each closed point $x\in X$ the maps
$\eta^*: E^i(U_x)\to E^i(\eta)$ are injective for $i=m$ and $i=m+1$, then the Cousin complex
$Cous(E,U_z;m)$ is exact.
\end{thm}

\begin{cor}\label{E_m_sheaf}
Under the hypothesis of Theorem \ref{cousin_1_z} the $p$-th cohomology group
\begin{equation}\label{Cousin_cohom}
H^p[0 \to E^q(\eta) \xrightarrow{\partial} \oplus_{y\in X^{(1)}} E^{q+1}_{y}(U_y) \xrightarrow{\partial} ... \xrightarrow{\partial}
\oplus_{x\in X^{(d+1)}}E^{q+d+1}_{x}(U_x) \to 0]
\end{equation}
coincides with the cohomology groups $H^p_{Zar}(X,\underline {E}^q)$, where $\underline {E}^q$ is the Zariski sheaf on $X$
associated with the presheaf $W\mapsto E^q(W)$.

If for each integer $i$ and each closed point $x\in X$ the map
$\eta^*: E^i(U_x)\to E^i(\eta)$ is injective, then there is a spectral sequence of the form
$H^p_{Zar}(X,\underline {E}^q)\Rightarrow E^{p+q}(X)$.
\end{cor}


\section{Cousin complexes for certain theories are exact}
The following result is {\it a very weak form of} \cite[Theorem, page 2]{GiP}.
However the proof of Theorem \ref{cousin_BW} uses only the Balmer--Witt functor,
whereas the main tool used in \cite{GiP} is the coherent Witt-groups developed by S.Gille \cite{Gi1}, \cite{Gi2}, \cite{Gi3}, \cite{Gi4}.
\begin{thm}\label{cousin_BW}[S.Gille,I.Panin \cite{GiP}]
Suppose $p\neq 2$. Let $(W,\partial)$ be the Balmer--Witt cohomology theory on $Sm'Op/V$. Then for each integer $m$
the Cousin complex $Cous(W,U,m)$ is exact. Particularly, the complex $Cous(W,U,0)$ is exact. Moreover, for each $z\in U^{(c)}$
there is a non-canonical isomorphism $W^{c}_z(U)\cong W(z)$.
So, a Gersten--Witt complex
\begin{equation}\label{G_Witt}
0\to W(U) \xrightarrow{\eta^*} W(\eta) \xrightarrow{\partial} \oplus_{y\in U^{(1)}} W(y)
\xrightarrow{\partial} ... \xrightarrow{\partial}
\oplus_{x\in U^{(d+1)}}W(x) \to 0
\end{equation}
is exact. Particularly, the map $W(\mathcal O)\to W(\mathcal K)$ is injective and the purity holds for the functor $W$.
Namely, if a class $[\phi]\in W(\mathcal K)$ can be lifted to the group $W(\mathcal O_y)$ for each point $y$ in $U^{(1)}$,
then the class $[\phi]$ can be lifted to the group $W(\mathcal O)$.
\end{thm}


\begin{defn}\label{defn:et_coh_theory}
Let $\mathcal F$ be an \'{e}tale sheaf on on the Big \'{e}tale site $(Sch'/V)_{Et}$.
Define a contravariant functor on $Sm'Op/V$ by
$$(Y,Y-Z)\mapsto H^*_Z(Y,\mathcal F):=Ext^n(\mathbb Z(Y)/\mathbb Z(Y-Z),\mathcal F),$$
where the Ext-groups are taken in the category of sheaves on the small \'{e}tale site $Y_{et}$.
The boundary maps $\partial_{(Y,Y-Z)}: H^n(Y-Z,\mathcal F)\to H^{n+1}_Z(Y,\mathcal F)$ is induced
by the short exact sequence
$0\to \mathbb Z(Y-Z) \to \mathbb Z(Y)\to \mathbb Z(Y)/\mathbb Z(Y-Z)\to 0$
of representable sheaves on the small \'{e}tale site $Y_{et}$.

Let $r>1$ be an integer such that for
any $Y\in Sch'/V$ one has $r\cdot \mathcal F(Y)=0$ and $r$ is coprime to $p$.
Then for each
$Y\in Sch'/V$ and each $n\geq 0$ and the projection $pr: Y\times \mathbb A^1\to Y$ the map
$pr^*: H^n_{et}(Y,\mathcal F)\to H^n_{et}(Y\times \mathbb A^1,\mathcal F)$ is an isomorphism.
Thus, the contravariant functor $(Y,Y-Z)\mapsto H^*_Z(Y,\mathcal F)$ together with the boundary maps
$\partial_{(Y,Y-Z)}$ form a cohomology theory on $Sm'Op/V$ in the sense of
\cite[Definition 2.1]{PW} (inspired by \cite[Definition 2.1]{PSm}).
\end{defn}

\begin{thm}\label{cousin_H_et_mu}
Suppose $r$ is coprime to the prime $p$ and $s$ be an integer. Let $(H^*_{et}(-, \mu^{\otimes s}_r),\partial)$
be the \'{e}tale cohomology theory on $Sm'Op/V$. Then for each integer $m$
the Cousin complex $Cous(H^*_{et}(-, \mu^{\otimes s}_r),U,m)$ is exact. Moreover, for each $z\in U^{(c)}$
there is a canonical isomorphism $H^{m+c}_{et,z}(U, \mu^{\otimes s}_r)=H^{m-c}_{et}(z,\mu^{\otimes (s-c)}_r)$.
So, the Bloch--Ogus type complex
\begin{equation}\label{Bloch__Ogus}
0\to H^{m}_{et}(U,\mu^{\otimes s}_r) \xrightarrow{\eta^*} H^{m}_{et}(\eta,\mu^{\otimes s}_r)
\xrightarrow{\partial} ... \xrightarrow{\partial}
\oplus_{x\in U^{(d+1)}}H^{m-d-1}_{et}(x,\mu^{\otimes (s-d-1)}_r) \to 0
\end{equation}
is exact.
\end{thm}

\begin{thm}\label{cousin_H_et_loc_const_F}
Suppose $r$ is coprime to the prime $p$.
Let $\cal F$ be a locally constant $r$-torsion
sheaf on the big \'{e}tale site $(Sch'/V)_{et}$.
Let $(H^*_{et}(-, \mathcal F),\partial)$
be the \'{e}tale cohomology theory on $Sm'Op/V$.
Then for each integer $m$
the Cousin complex $Cous(H^*_{et}(-, \mathcal F),U,m)$ is exact. Moreover, for each $z\in U^{(c)}$
there is a canonical isomorphism $H^{m+c}_{et,z}(U, \mathcal F)=H^{m-c}_{et}(z,\mathcal F(-c))$.
So, the Bloch--Ogus type complex
\begin{equation}\label{Bloch__Ogus}
0\to H^{m}_{et}(U,\mathcal F) \xrightarrow{\eta^*} H^{m}_{et}(\eta,\mathcal F)
\xrightarrow{\partial} ... \xrightarrow{\partial}
\oplus_{x\in U^{(d+1)}}H^{m-d-1}_{et}(x,\mathcal F(-d-1)) \to 0
\end{equation}
is exact.
\end{thm}


Consider the Thomason-Throughbor $K$-groups with finite coefficients $\mathbb Z/r$.
They form {\it a cohomology theory} on the category $Sm'Op/V$ in the sense of \cite{PSm}.
Namely, for each integer $n$ and each pair $(X,X-Z)$ in $Sm'Op/V$ put
$K^{-n}(X,X-Z;\mathbb Z/r)=K_{n}(X on Z;\mathbb Z/r)$ (see \cite{TT}).
The definition of $\partial$ is contained in \cite[Theorem 5.1]{TT}
Repeating literally arguments of
\cite[Example 2.1.8]{PSm} we see that that this way we get a cohomology theory
on $Sm'Op/V$. Moreover, if $X$ is quasi-projective then
$K_n(X on X; \mathbb Z/r)$ coincides with the Quillen's $K$-groups $K^Q_n(X;\mathbb Z/r)$ by
\cite[Theorems 3.9 and 3.10]{TT}. Write $K^{-n}_Z(X;\mathbb Z/r)$ for $K^{-n}(X,X-Z;\mathbb Z/r)$.

\begin{thm}\label{cousin_K_w_Z_rZ}[Gillet, H., Levine M. \cite{GL}]
Suppose $r$ is coprime to the prime $p$. Let $(K^*(-;\mathbb Z/r),\partial)$ be the Thomason-Throughbor $K$-theory with finite coefficients $\mathbb Z/r$.
It is a cohomology theory on $Sm'Op/V$  as explained just above.
Then for each integer $m$
the Cousin complex $Cous(K^*(-;\mathbb Z/r),U,m)$ is exact.
Moreover, for each $z\in U^{(c)}$
there is a canonical isomorphism
$K^{m+c}_z(U;\mathbb Z/r)=K^{m+c}(z;\mathbb Z/r)=K_{-m-c}(z;\mathbb Z/r)$.
So, the Gersten complex
\begin{equation}\label{thm:Gersten_w_Z/r}
0\to K_m(U;\mathbb Z/r) \xrightarrow{\eta^*} K_m(\eta;\mathbb Z/r) \xrightarrow{\partial} ... \xrightarrow{\partial}
\oplus_{x\in U^{(d+1)}}K_{m-d-1}(x;\mathbb Z/r) \to 0
\end{equation}
is exact.
\end{thm}

\begin{thm}\label{cousin_K2}
Suppose $r$ is coprime to the prime $p$.
The following Gersten type complexes are exact, where $K_*$ is the Quillen K-theory,
\begin{equation}\label{thm:Gersten_is_Exact}
0\to K_2(U) \xrightarrow{\eta^*} K_2(\eta) \xrightarrow{\partial} \oplus_{y\in U^{(1)}} K_1(k(y)) \xrightarrow{\partial}
\oplus_{x\in U^{(2)}}K_0(k(x)) \to 0
\end{equation}
\begin{equation}\label{thm:Gersten_mod_n_is_Exact}
0\to K_2(U)/r \xrightarrow{\eta^*} K_2(\eta)/r \xrightarrow{\partial} \oplus_{y\in U^{(1)}}K_1(k(y))/r \xrightarrow{\partial}
\oplus_{x\in U^{(2)}}K_0(k(x))/r \to 0
\end{equation}
\end{thm}

\section{Suslin's exact sequence in mixed characteristic}

The following result gives a mixed characteristic version of the Suslin exact sequence.
\begin{thm}\label{thm:Suslin_exact}
Suppose $r$ is coprime to the prime $p$.
Then there is an exact sequence of the form
\begin{equation}
\label{Suslin_sequence_2}
0\to H^1_{Zar}(X,\underline K_2)/r \xrightarrow{\alpha} NH^{3}_{et}(X,\mu^{\otimes 2}_r) \xrightarrow{\beta} {_{r}}H^2_{Zar}(X,\underline K_2)\to 0,
\end{equation}
where $NH^3_{et}(X,\mu^{\otimes 2}_r)=ker[H^3_{et}(X,\mu^{\otimes 2}_r)\xrightarrow{\eta^*} H^3_{et}(\eta,\mu^{\otimes 2}_r)]$.
\end{thm}
\begin{thm}\label{NH3=torsion_CH}
Suppose additionally that the ring $A$ is henzelian, the residue field of $A$ is finite and let $l$ be a prime different of $p$.
Suppose this time that $X$ is $A$-smooth projective irreducible and has relative dimension $d=2$. Then \\
a) the map ${_{\{l\}}}H^2_{Zar}(X,\underline K_2)\to {_{\{l\}}}CH^2(X_v)$ is an isomorphism;\\
b) the map $\beta: NH^{3}_{et}(X,\mathbb Q_l/\mathbb Z_l(2))\to {_{\{l\}}}H^2_{Zar}(X,\underline K_2)$ is an isomorphism;\\
c) the map $NH^{3}_{et}(X,\mathbb Q_l/\mathbb Z_l(2))\to NH^{3}_{et}(X_v,\mathbb Q_l/\mathbb Z_l(2))$ is an isomorphism;\\
d) the group ${_{\{l\}}}H^2_{Zar}(X,\underline K_2)$ is finite.
\end{thm}

\begin{rem}\label{connectivity}
Since $X$ is smooth projective over $V$ and irreducible and $A$ is henzelian it follows
that the closed fibre $X_v$ is
{\it irreducible too}.
To prove this one has to use the Steiner
decomposition of the morphism $p: X\to V$ as $X\xrightarrow{q} W\xrightarrow{\pi} V$,
where $q$ is surjective with connected fibres and $\pi$ is finite surjective.
Since $q$ is projective and surjective and $X$ is irreducible it follows that
$W$ is irreducible. Since $V$ is local henzelian and $\pi$ is finite it follows
that $W$ has a unique closed point, say $w$. Since the fibres of $q$ are connected
it follows that $X_v=p^{-1}(v)=q^{-1}(w)$ is connected. Since $X_v$ is $v$-smooth,
hence $X_v$ is irreducible.
\end{rem}


The following colollary is a version of the Roitman theorem.
\begin{cor}\label{thm:Roitman}
Let $A$,$X$,$d$,$l$ be as in Theorem \ref{NH3=torsion_CH}. Suppose the residue field of $A$ is an algebraic closure of the finite field $\mathbb F_p$.
Then the natural map
$${_{\{l\}}}H^2_{Zar}(X,\underline K_2)\to {_{\{l\}}}Alb(X/V)$$
is an isomorphism. There are equalities
$NH^{3}_{et}(X_v,\mathbb Q_l/\mathbb Z_l(2))=H^{3}_{et}(X_v,\mathbb Q_l/\mathbb Z_l(2))$
and
$NH^{3}_{et}(X,\mathbb Q_l/\mathbb Z_l(2))=H^{3}_{et}(X,\mathbb Q_l/\mathbb Z_l(2))$
and the map $\beta$ is an isomorphism
$$\beta: H^{3}_{et}(X,\mathbb Q_l/\mathbb Z_l(2))=NH^{3}_{et}(X,\mathbb Q_l/\mathbb Z_l(2))\to {_{\{l\}}}H^2_{Zar}(X,\underline K_2).$$
\end{cor}

\begin{thm}\label{NH3=torsion_CH_2}
Let $A$,$X$,$l$ be as in Theorem \ref{NH3=torsion_CH}, but this time $d=dim_V X$ is arbitrary.
Suppose the residue field of $A$ is an algebraic extension of the finite field $\mathbb F_p$.
Then \\
a) the map $\beta: NH^{3}_{et}(X,\mathbb Q_l/\mathbb Z(2))\to {_{\{l\}}}H^2_{Zar}(X,\underline K_2)$ is an isomorphism;\\
b) the map $NH^{3}_{et}(X,\mathbb Q_l/\mathbb Z(2))\to NH^{3}_{et}(X_v,\mathbb Q_l/\mathbb Z(2))$ is injecive;\\
c) the map ${_{\{l\}}}H^2_{Zar}(X,\underline K_2)\to {_{\{l\}}}CH^2(X_v)$ is injective;\\
d) if the residue field $\bar A$ of $A$ is finite, then the group ${_{\{l\}}}H^2_{Zar}(X,\underline K_2)$ is finite.
\end{thm}

\section{Rationally isotropic quadratic spaces}
The results of this section
extend to the mixed characteristic case isotropy and purity theorems for {\it quadratic spaces}
proven in \cite{P}, \cite{PP1}, \cite{PP2}.
Following the method of \cite{PP2} these results are derived in this section
using
Corollary \ref{Homotopy_w_Traces} and Theorem \ref{cousin_BW}.
A proof of Theorem \ref{cousin_BW} is given in Section \ref{sect:cousin_BW}.

Let $A$, $p>0$, $d\geq 1$, $X$, $x_1,x_2,\dots,x_n\in X$, $\mathcal O$ and $U$  be as in Section \ref{agreements}.
We suppose in this section that $1/2$ is in $A$. Write $\mathcal K$ for the fraction field of the ring $\mathcal O$.


We refer to \cite{P} for definitions of a quadratic space and of an isotropic quadratic space.
Here we just indicate the following: if $(Q,\psi)$ is a quadratic space over a Dedekind semi-local domain $R$ with
a fraction field $L$ and $(Q,\psi)\otimes_R L$ is isotropic over $L$, then $(Q,\psi)$ is isotropic.

\begin{thm}[Rationally isotropic spaces are locally isotropic]
\label{Rat_Iso_Loc_Iso}
(1) Suppose each irreducible component of $X_v$ contains at least one of the point $x_i$'s.
Let $(P,\varphi)$ be a quadratic space over $\mathcal O$. If
$(P,\varphi)\otimes_{\mathcal O} \mathcal K$ is isotropic over $\mathcal K$, then
$(P,\varphi)$ is isotropic over $\mathcal O$, that is there exists a unimodular vector $v\in V$
with $\varphi (v)=0$.

(2) If $X_v$ is irreducible, then the hypotheses of the item (1) hold automatically. So,
each quadratic space over $\mathcal O$, which is isotropic over
$\mathcal K$, is isotropic over $\mathcal O$.
\end{thm}

To state the first corollary of the theorem we need to recall
the notion of unramified spaces.
Let $R$ be a Noetherian domain and $L$ be its fraction field.
Recall that a quadratic space $(Q, \psi)$ over $L$
is {\it unramified over} $R$ \ if for every height one prime ideal
$\wp$ of $R$ there exists a quadratic space
$(Q_{\mathfrak p},\varphi_{\wp})$
over $R_{\wp}$ such that
the spaces
$(Q_{\wp},\varphi_{\wp}) \otimes_{R_{\wp}} L$ and $(Q,\psi)$ are isomorphic.

\begin{cor}[Purity for quadratic spaces]
\label{Purity_Q}
(1) Suppose each irreducible component of $X_v$ contains at least one of the point $x_i$'s.
Let $(Q,\psi)$ be a quadratic space over $\mathcal K$ which is unramified over
$\mathcal O$. Then there exists a quadratic space $(P,\varphi)$ over $\mathcal O$
extending the space $(Q,\psi)$, that is the spaces
$(P,\varphi)\otimes_{\mathcal O} \mathcal K$ and $(Q,\psi)$ are isomorphic.

(2) If $X_v$ is irreducible, then the hypotheses of the item (1) hold automatically. So,
each quadratic space over $\mathcal K$, which is unramified over
$\mathcal O$, can be lifted to a quadratic space over $\mathcal O$.
\end{cor}

\begin{cor}\label{Pres_of_a_unit}
(1) Suppose each connected component of $X_v$ contains at least one of the point $x_i$'s.
Let $(P,\varphi)$ be a quadratic space over $\mathcal O$ and let $u \in \mathcal O^{\times}$ be
a unit. Suppose the equation $\varphi = u$ has a solution over $\mathcal K$
then it has a solution over $\mathcal O$, that is there exists a
vector $v \in V$ with $\varphi (v)= u$.

(2) If $X_v$ is irreducible, then the hypotheses of the item (1) hold automatically. So,
if the equation $\varphi = u$ has a solution over $\mathcal K$
then it has a solution over $\mathcal O$.
\end{cor}

\begin{notation}\label{inj_to_gen_point}
Let $X_{v,j}\subseteq X_v$ ($j$ runs from $1$ to $l$) be the irreducible components of the scheme $X_v$. Let $\eta_j$ be its generic point.
Write $\mathcal O_{X,X_v}$ for $\mathcal O_{X,\eta_1,...,\eta_l}$. If $X'\subset X$ is an open containing all the points $\eta_j$'s,
then $\mathcal O_{X,\eta_1,...,\eta_l}=\mathcal O_{X',\eta_1,...,\eta_l}$. So, using short notation
$\mathcal O_{X,X_v}=\mathcal O_{X',X'_v}$. Put $\mathcal S=Spec \mathcal \ O_{X,X_v}=Spec \ \mathcal O_{X',X'_v}=\mathcal S'$.
\end{notation}

\begin{proof}[Reducing Theorem \ref{Rat_Iso_Loc_Iso} to Theorem \ref{cousin_BW} and Corollary \ref{Homotopy_w_Traces_2}]
First consider the case when the residue field $k(v)$ of $A$ is infinite.
We may assume that there is an affine open neighborhood $X'$ of all the points $x_i$'s
and a quadratic space $(P',\varphi')$ over $X'$ which extends to $X'$ the space $(P,\varphi)$.
Moreover, this $X'$ can be chosen such that it
contains all generic points of the scheme $X_v$, but it
does not contain any irreducible component of the scheme $X_v$.
By the hypothesis of the theorem the quadratic space $(P',\varphi')$ is isotropic over $\mathcal K$.
Thus, it is isotropic over the semi-local Dedekind ring $\mathcal O_{X',X'_v}$,
where the ring $\mathcal O_{X',X'_v}$ is defined in
Notation \ref{inj_to_gen_point}.
Hence there is an $f\in \Gamma(X',\mathcal O_{X'})$ such that
the space $(P',\varphi')|_{X'_f}$ is isotropic and
the closed subset $Y:=\{f=0\}$ of $X'$ does not contain any irreducible component of the scheme $X'_v$.

Now consider the diagram from Corollary \ref{Homotopy_w_Traces_2}. By the item (5) of the Corollary
the quadratic spaces $p^*_X(P',\varphi')$ and $p^*_U(P,\varphi)$ are isomorphic.
Thus, by Corollary \ref{Homotopy_w_Traces_2} the pull-backs of the quadratic space $(P,\varphi)$ as to
$\mathcal D_1$, so to $\mathcal D'_0$ are both isotropic.

The schemes $\mathcal X$ and $\Aff^1_U$ are both regular and irreducible.
The morphism $\tau$ is finite surjective. By \cite[Corollary 18.17]{E}
the $\mathcal O[t]$-module
$\Gamma(\mathcal X,\mathcal O_{\mathcal X})$
is finitely generated projective (of constant rank).
By the statements (1) and (4) of the Corollary \ref{Homotopy_w_Traces_2} one has an equality
$[\mathcal D_1:U]=1+[\mathcal D'_0:U]$.
Hence one of the degree's
$[\mathcal D_1:U]$ or $[\mathcal D'_0:U]$ is an {\it odd number}.
By the statement (3) of the Corollary \ref{Homotopy_w_Traces_2}
as $\mathcal D_1$, so $\mathcal D'_0$ are both finite \'{e}tale over $U$.
Applying now \cite[Theorem 1.1]{PP1}
we conclude that the quadratic space $(P,q)$ is isotropic.

Now consider the case when the residue field $k(v)$ of $A$ is finite.
Let $l$ be an odd prime number different of $p$.
Consider a tower of finite \'{e}tale extensions
$$A\subset A_1\subset A_2\subset A_3 ... \subset A_{\infty}$$
such that each $A_m$ is a d.v.r., each extension  $A_m\subset A_{m+1}$
is finite \'{e}tale of degree $l$,
for each point $x_i$'s the $k(x_i)$-algebra $k(x_i)\otimes_A A_m$ is a field.
Replace $A$ with $A_{\infty}$, $X$ with $X_{\infty}=X\otimes_A A_{\infty}$, points
$x_1,x_2,\dots,x_n\in X$ with points $x_{1,\infty},x_{2,\infty},\dots, x_{n,\infty}\in X_{\infty}$,
$\mathcal O$ with $\mathcal O_{\infty}=\mathcal O\otimes_A A_{\infty}$
and $U$ with $U_{\infty}=U\otimes_A A_{\infty}$.
Write $\mathcal K_{\infty}$ for $\mathcal K\otimes_A A_{\infty}$.
Write $(P,\varphi)_{\infty}$ for the quadratic space
$(P,\varphi)\otimes_A A_{\infty}$
over $\mathcal O_{\infty}$.
Since $X$ is $A$-smooth it follows that $X_{\infty}$ is $A_{\infty}$-smooth. Thus,
$X_{\infty}$ and $U_{\infty}$ are regular schemes.
Since $U$ is irreducible, $U_{\infty}$ is regular and $x_{1,\infty},x_{2,\infty},\dots, x_{n,\infty}$
are all its closed points it follows that $U_{\infty}$ is irreducible. Thus,
$\mathcal K_{\infty}$ is a field (it is the fraction field of
$\mathcal O_{\infty}$).

By the first part of the proof the the quadratic space
$(P,\varphi)\otimes_A A_{\infty}$
is isotropic over
$\mathcal O_{\infty}$.
Thus, there exists an integer $m>0$ such that
the quadratic space
$(P,\varphi)\otimes_A A_m$
is isotropic over
$\mathcal O_{m}=\mathcal O\otimes_A A_m$.
The extension $\mathcal O\subset \mathcal O_{m}$
is a finite \'{e}tale extension of odd degree.
Applying now \cite[Theorem 5.1]{Sc}
we conclude that the quadratic space $(P,q)$ is isotropic.

\end{proof}

\begin{proof}[Reducing Theorem \ref{Purity_Q} to Theorems \ref{cousin_BW} and \ref{Rat_Iso_Loc_Iso}]
By Theorem \ref{cousin_BW} there exist a
quadratic space $(V,\varphi)$ over $\mathcal O$ and an integer $n \geq 0$
such that $(P,\varphi) \otimes_{\mathcal O} \mathcal K \cong (Q,\psi) \perp \mathbb H^{n}_{\mathcal K}$,
where $\mathbb H_{\mathcal K}$ is a hyperbolic plane. If $n > 0$ then
the space $(V,\varphi) \otimes_{\mathcal O} \mathcal K$ is isotropic. By the Theorem
\ref{Rat_Iso_Loc_Iso}
the space $(V,\varphi)$ is isotropic too. Thus
$(V,\varphi) \cong (V^{\prime}, \varphi^{\prime}) \perp \mathbb H_{\mathcal O}$
for a quadratic space
$(V^{\prime}, \varphi^{\prime})$ over $\mathcal O$.
Now Witt's Cancellation theorem over a field
\cite[Chap.I, Thm.4.2]{La}
shows that
$(V^{\prime}, \varphi^{\prime}) \otimes_{\mathcal O} \mathcal K \cong (W,\psi) \perp \mathbb H^{n-1}_{\mathcal K}$.
Repeating
this procedure several times we may assume that $n=0$, which means
that $(V,\varphi) \otimes_\mathcal O \mathcal K \cong (W,\psi)$.
\end{proof}

\begin{proof}[Reducing Corollary \ref{Pres_of_a_unit} to Theorem \ref{Rat_Iso_Loc_Iso}]
Let $(\mathcal O,-u)$ be a the rank one quadratic space over $\mathcal O$
corresponding the unit $-u$. The space
$(V,\varphi)_\mathcal K \perp (\mathcal K,-u)$ is isotropic. Thus, the space
$(V,\varphi) \perp (\mathcal O,-u)$
is isotropic by Theorem \ref{Rat_Iso_Loc_Iso}. By the Claim below
there exists
a vector $v \in V$ with $\varphi (v)= u$.
Clearly $v$ is unimodular.

{\it Claim.}
Let $(W, \psi)=(V,\phi) \perp (R,-u)$. The space $(W, \psi)$
is isotropic if and only if there exists
a vector $v \in V$ with $\varphi (v)= u$.

This Claim is proved in
\cite[the proof of Proposition 1.2]{C-T}.
\end{proof}

\section{Grothendieck--Serre conjecture for $SL_{1,D}$ }
Let $A$, $p>0$, $d\geq 1$, $X$, $x_1,x_2,\dots,x_n\in X$, $\mathcal O$ and $U$  be as in Section \ref{agreements}.
Write $\mathcal K$ for the fraction field of the ring $\mathcal O$.

\begin{thm}[Grothendieck--Serre conjecture for $SL_{1,D}$]
\label{Gr_Serre_for Azumaya}
Let $D$ be an Azumaya $\mathcal O$-algebra
and $Nrd_D: D^{\times}\to \mathcal O^{\times}$ be the reduced norm homomorphism.
Let $a \in \mathcal O^{\times}$. If $a$ is a reduced norm for the central simple
$\mathcal K$-algebra $D\otimes_{\mathcal O} \mathcal K$,
then $a$ is a reduced norm for the algebra $D$.
\end{thm}

The following result is the "constant" case of Theorem \ref{Gr_Serre_for Azumaya}. Its proof
is easier than the one of Theorem \ref{Gr_Serre_for Azumaya}. This is why we decided
to include it in the paper.
\begin{thm}[Grothendieck--Serre conjecture for $SL_{1,\mathcal A}$]
\label{Gr_Serre_for Azumaya_const}
Let $\mathcal A$ be an Azumaya $A$-algebra.
Let $a \in \mathcal O^{\times}$. If $a$ is a reduced norm for the central simple
$\mathcal K$-algebra $\mathcal A\otimes_{A} \mathcal K$,
then $a$ is a reduced norm for the Azumaya $\mathcal O$-algebra $\mathcal A\otimes_A \mathcal O$.
\end{thm}

\begin{prop}\label{Purity_for_K_Adzum}
Let $\mathcal A$ be an Azumaya algebra over $A$.
Let $K_*$ be the Quillen $K$-functor. Then for each integer $n\geq 0$ one has an exact sequence
\begin{equation}
K_n(\mathcal A\otimes_A \mathcal O)\to K_n(\mathcal A\otimes_A \mathcal K)\xrightarrow{\partial} \oplus_{y\in U^{(1)}}K_{n-1}(\mathcal A\otimes_A k(y)).
\end{equation}
Particularly, this sequence is exact for $n=1$.
\end{prop}

\begin{prop}\label{Purity_for Azumaya}
Let $D$ be an Azumaya $\mathcal O$-algebra.
Let $K_*$ be the Quillen $K$-functor.
Then for each integer $n\geq 0$
the sequence is exact
\begin{equation}
K_n(D)\to K_n(D\otimes_{\mathcal O} \mathcal K)\xrightarrow{\partial} \oplus_{y\in U^{(1)}}K_{n-1}(D\otimes_{\mathcal O} k(y))
\end{equation}
Particularly, it is exact for $n=1$.
\end{prop}

\begin{rem}\label{two_propositions}
Proposition \ref{Purity_for_K_Adzum} is a particular case of the item (b) of Theorem \ref{General_Purity}.
This is the case since the Thomason-Throughbor $K$-groups with coefficients in $\mathcal A$
form {\it a cohomology theory} on the category $SmOp/V$ in the sense of \cite{PSm}.
Details are given below in this section.

However, the Thomason-Throughbor $K$-groups with coefficients in $D$ does not form
a cohomology theory on the category $SmOp/V$ in the sense of \cite{PSm} (indeed,
the algebra $D$ does not come from $V$). This show that a proof of Proposition
\ref{Purity_for Azumaya} requires a different method (see below in this section).
\end{rem}

\begin{proof}[Reducing Theorem \ref{Gr_Serre_for Azumaya_const} to Proposition \ref{Purity_for_K_Adzum}]
Consider a commutative diagram of groups
\begin{equation}
\label{RactangelDiagram}
    \xymatrix{
K_1(\mathcal A\otimes_A \mathcal O) \ar[rr]^{\eta^*} \ar[d]_{Nrd}&& K_1(\mathcal A\otimes_A \mathcal K)\ar[rr]^{\partial} \ar[d]^{Nrd} &&
\oplus_{y\in U^{(1)}}K_0(\mathcal A\otimes_A k(y)) \ar[d]^{Nrd}  \\
\mathcal O^{\times} \ar[rr]^{\eta^*}&& \mathcal K^{\times} \ar[rr]^{\partial} && \oplus_{y\in U^{(1)}}K_0(k(y))\\ }
\end{equation}
By Proposition \ref{Purity_for_K_Adzum} the complex on the top is exact. The bottom map $\eta^*$ is injective.
The right hand side vertical map $Nrd$ is injective. Thus, the map
$$\mathcal O^{\times}/Nrd(K_1(\mathcal A\otimes_A \mathcal O))\to \mathcal K^{\times}/Nrd(K_1(\mathcal A\otimes_A \mathcal K))$$
is injective.
The image of the left vertical map coincides with
$Nrd((\mathcal A\otimes_A \mathcal O)^{\times})$
and the image of the middle vertical map coincides with
$Nrd((\mathcal A\otimes_A \mathcal K)^{\times})$. Thus, the map
$$\mathcal O^{\times}/Nrd((\mathcal A\otimes_A \mathcal O)^{\times})\to \mathcal K^{\times}/Nrd((\mathcal A\otimes_A \mathcal K)^{\times})$$
is injective. The derivation of Theorem \ref{Gr_Serre_for Azumaya_const} from Proposition
\ref{Purity_for_K_Adzum} is completed.
\end{proof}

\begin{proof}[Proof of Proposition \ref{Purity_for_K_Adzum}]
There are two ways of proving this Proposition. To use the Thomason-Throughbor $K$-groups with coefficients in $\mathcal A$
or to use Quillen's $K'$-groups with coefficients in $\mathcal A$.
We prefer to use the Thomason-Throughbor $K$-groups with coefficients in $\mathcal A$.
They form {\it a cohomology theory} on the category $SmOp/V$ in the sense of \cite{PSm}.
Namely, for each integer $n$ and each pair $(X,X-Z)$ in $SmOp/V$ put
$E^n(X,X-Z)=K_{-n}(X on Z;\mathcal A)$ (see \cite{TT}).
The definition of $\partial$ is contained in \cite[Theorem 5.1]{TT}.
Repeating literally arguments of
\cite[Example 2.1.8]{PSm} we see that this way we get a cohomology theory
on $SmOp/V$. Moreover, If $X$ is quasi-projective then
$K(X on X; \mathcal A)$ coincides with the Quillen's $K$-groups $K^Q_n(X)$ by
\cite[Theorems 3.9 and 3.10]{TT}.
Now our proposition follows from the item (b) of Theorem \ref{General_Purity}.
\end{proof}

\begin{proof}[Reducing Theorem \ref{Gr_Serre_for Azumaya} to Proposition \ref{Purity_for Azumaya}]
Consider the commutative diagram of groups
\begin{equation}
\label{RactangelDiagram}
    \xymatrix{
K_1(D) \ar[rr]^{\eta^*} \ar[d]_{Nrd}&& K_1(D\otimes_{\mathcal O} \mathcal K)\ar[rr]^{\partial} \ar[d]^{Nrd} &&
\oplus_{y\in U^{(1)}}K_0(D\otimes_{\mathcal O} k(y)) \ar[d]^{Nrd}  \\
\mathcal O^{\times} \ar[rr]^{\eta^*}&& \mathcal K^{\times} \ar[rr]^{\partial} && \oplus_{y\in U^{(1)}}K_0(k(y))\\ }
\end{equation}
By Proposition \ref{Purity_for Azumaya} the complex on the top is exact. Repeat now the arguments
from reducing
Theorem \ref{Gr_Serre_for Azumaya_const} to Proposition
\ref{Purity_for_K_Adzum}.
These will complete reducing the theorem to
Proposition \ref{Purity_for Azumaya}.
\end{proof}

\begin{proof}[Proof of Proposition \ref{Purity_for Azumaya}]
To prove this proposition
it sufficient to prove vanishing of the support extension map
$ext_{2,1}: K'_{n}(U;D)_{\geq 2}\to K'_{n}(U;D)_{\geq 1}$.
Prove that $ext_{2,1}=0$.
Take an $a\in K'_n(U;D)_{\geq 2}$.
We may assume that $a\in K'_n(Z;D)$ for a closed $Z$ in $U$ with $codim_U(Z)\geq 2$.
Enlarging $Z$ we may assume that each its irreducible component
contains at least one of the point $x_i$'s and still $codim_U(Z)\geq 2$.
Our aim is to find a closed subset $Z_{ext}$ in $U$ containing $Z$ such that
the element $a$ vanishes in $K'_n(Z_{ext};D)$ and
$codim_U(Z_{ext})\geq 1$. We will follow the method of \cite{PS}.

One can find closed points  $x_{n+1},..., x_M$ in $X$ such that \\
(a) each irreducible component of $X_v$ contains at least one of the point $x_j$'s ($j\in \{1,...,M\}$);\\
(b) for the ring
$\mathcal O'=\mathcal O_{X,x_1,..., x_M}$
and the scheme
$U'=Spec~\mathcal O'$
the set $Z$ is closed in $U'$.
Let $X^*$ be a neighborhood of the points
$x_j$'s ($j\in \{1,...,M\}$)
(containing all generic points of the scheme $X_v$,
but not containing any irreducible component of the scheme $X_v$).
Let
$Z^*$ be the closure of $Z$ in $X^*$.

By Theorem \ref{geom_pres_mixed_char} one can find inside $X^*$
an affine neighborhood $X^{\circ}$ of points
$x_1,x_2,\dots,x_M$, an open affine subscheme $S\subseteq \Pro^{d-1}_A$ and a flat $A$-morphism \\
$$q: X^{\circ}\to S$$
such that $q$ is smooth in a neighborhood of $q^{-1}(S_v)$ and
$Z^{\circ}/S$ is finite, where $Z^{\circ}=Z^*\cap X^{\circ}$.

Now take back to the points $x_1,...,x_n$, the scheme $U$ and the Azumaya $\mathcal O$-algebra $D$.
Put $s_i=q(x_i)$. Consider the semi-local ring
$\mathcal O_{S,s_1,...,s_n}$, put $B=Spec \mathcal O_{S,s_1,...,s_n}$ and
$X_B=q^{-1}(B)\subset X^{\circ}$. Put $Z_B=Z^{\circ}\cap X_B$.
Write $q_B$ for $q|_{X_B}: X_B\to B$. Note that $Z_B$ is finite over $B$.
Since $B$ is semi-local, hence so is $Z_B$. Also, $Z_B$ contains all the points
$x_1,...,x_n$. Hence any neighborhood of all the point $x_1,...,x_n$ in $X_B$
contains $Z_B$.
These show that $Z_B=Z$ and $U$ is in $X_B$. Since now on write $Z$ for $Z_B$.

Replacing $X_B$ with an appropriate neighborhood of $U$ in $X_B$ we may assume that
the Azumaya algebra $D$ is an Azumaya algebra over $X_B$. Since $Z\subset U$ and $Z$ is finite over $S$
it is closed also in the chosen neighborhood of $U$. We will still write $X_B$ for
the chosen neighborhood. Recall that we are given with the element $a\in K'_n(Z;D)$.

Now put $\mathcal Z=Z\times_B X_B$. Let
$\Pi: \mathcal Z\to X_B$
be the projection to $X_B$
and
$p_Z: \mathcal Z\to Z$ be the projection to $Z$.
The closed embedding $in: Z\hookrightarrow X_B$ defines a section
$\Delta=(id\times in): Z\to \mathcal Z$
of the projection $p_Z$. Also one has $in=\Pi \circ \Delta$.
Put $D_Z=in^*(D)$ and write ${_{Z}}D_X$ for the Azumaya algebra
$\Pi^*(D)\otimes p^*_Z(D^{op}_Z)$
over $\mathcal Z$.
The $\mathcal O_{\mathcal Z}$-module
$\Delta_*(D_Z)$
has an obvious left ${_{Z}}D_X$-module structure.
And it is equipped with an obvious epimorphism
$\pi: {_{Z}}D_X\to \Delta_*(D_Z)$
of the left ${_{Z}}D_X$-modules.
Following \cite{PS} one can see that $I:=Ker(\pi)$
is a left projective ${_{Z}}D_X$-module. Hence the left ${_{Z}}D_X$-module
$\Delta_*(D_Z)$ defines an element $[\Delta_*(D_Z)]=[{_{Z}}D_X]-[I]$
in $K_0(\mathcal Z;{_{Z}}D_X)$. This element has rank zero.
Hence by \cite{} it vanishes semi-locally on $\mathcal Z$.

Since $Z/B$ is finite the morphism $\Pi$ is finite.
Hence any neighborhood of $Z\times_B Z$ in $\mathcal Z$
contains a neighborhood of the form $Z\times_B W$,
where $W\subset X_B$ is an open containing $Z$.
Write $\mathcal Z_W$ for $Z\times_B W$.
The set $Z$ is finite over $B$ and semi-local.
Hence $Z\times_B Z$ is finite over the first coordinate
and semi-local. Hence there is an open $W$ in $X_B$ containing $Z$ such that
$$0=[\Delta_*(D_Z)]|_{\mathcal Z_W}\in K_0(\mathcal Z_W;{_{Z}}D_X).$$
To simplify notation we will write below in this proof $X_B$ for this $W$,
$\mathcal Z$ for this $\mathcal Z_W$, $\Pi$ for $\Pi|_{\mathcal Z_W}: \mathcal Z_W\to W$ and
$q_Z: \mathcal Z\to Z$ for the projection $\mathcal Z_W\to Z$.

Since $\Pi$ is a finite morphism.
Hence the image
$\Pi(\mathcal Z_{red})$ is closed in $X_B$.
Write $Z_{new}$ for $\Pi(\mathcal Z_{red})$ in $X_B$.
Clearly, we have closed inclusions $Z\subset Z_{new}\subset X_B$.
The nearest aim is to show that for the inclusion
$i: Z\hookrightarrow Z_{new}$
one has $i_*(a)=0$ in $K'_n(Z_{new};D)$.
This will allow to put $Z_{ext}=Z_{new}\cap U$
and to show that $a$ vanishes in $K'_n(Z_{ext};D)$.
Recall that the morphism $q_B: X_B\to B$ is flat. Hence so is the morphism $p_Z$.

 The functor $(P,M)\mapsto P\otimes_{p^*_Z(D_Z)} M$
induces a bilinear pairing
$$\cup_{q^*_Z(D_Z)}: K_0(\mathcal Z;{_{Z}}D_X)\times K_n(\mathcal Z;q^*_Z(D_Z))\to K_n(\mathcal Z;q^*_X(D))$$
Each element
$\alpha\in K_0(\mathcal Z;{_{Z}}D_X)$
defines a group homomorphism
$$\alpha_*=\Pi_*\circ ( \alpha \cup_{q^*_Z(D_Z)} - )\circ q^*_Z: K'_n(Z;D)=K'_n(Z;D_Z)\to K'_n(Z_{new};D)$$
which takes an element $b\in K'_n(Z;D_Z)$ to the one
$b\mapsto \Pi_*(\alpha \cup_{q^*_Z(D_Z)} q^*_Z(b))$
in $K'_n(Z_{new};D)$.

Following \cite{PS} we see that the map $[\Delta_*(D_Z)]_*$ coincides with the map
$i_*: K'_n(Z;D)=K'_n(Z;D_Z)\to K'_n(Z_{new};D)$. The equality
$0=[\Delta_*(D_Z)]\in K_0(\mathcal Z;{_{Z}}D_X)$ proven just above
shows that the map $i_*$ vanishes.

Since $Z$ is in $U\cap Z_{new}=Z_{ext}$, hence we have a closed inclusion $in: Z\hookrightarrow Z_{ext}$.
The inclusion $U\xrightarrow{can} X_B$ is a flat morphism. Hence the inclusion
$inj: Z_{ext}=Z_{new}\cap U\to Z_{new}$ is also a flat morphism. Thus the map
$inj^*: K'_n(Z_{new};D)\to K'_n(Z_{ext};D)$
is well-defined. Moreover $inj^*\circ i_*=in_*$.
The map $i_*$ vanishes, hence the map $in_*$ vanishes also.
Particularly, $in_*(a)=0$.
The proposition is proved.
\end{proof}

\section{Proofs of Theorems \ref{General_Purity}, \ref{cousin_1}, \ref{General_Purity_z}, \ref{cousin_1_z}.}
\begin{proof}[Reducing Theorem \ref{General_Purity} to Theorem \ref{An extended_moving_lemma}]
The nearest aim is to prove the assertions (c) and (d).
So, we must prove the vanishing of the support extension map
$$ext_{2,1}: E^m_{\geq 2}(U)\to E^m_{\geq 1}(U).$$
Prove that $ext_{2,1}=0$.
Take an $a\in E^m_{\geq 2}(U)$. We may assume that $a\in E^m_Z(U)$ for a closed $Z$ in $U$ with
$codim_U(Z)\geq 2$.
One can find closed points  $x_{n+1},..., x_M$ in $X$ such that \\
(a) each irreducible component of $X_v$ contains at least one of the point $x_j$'s ($j\in \{1,...,M\}$);\\
(b) for the ring
$\mathcal O''=\mathcal O_{X,x_1,..., x_M}$
and the scheme
$U''=Spec~\mathcal O''$
the set $Z$ is closed in $U''$.
By the excision property for the open embedding $in: U\hra U''$ the pull-back
$$in^*: E^m_Z(U'')\to E^m_Z(U)$$
is an isomorphism. Let $a''\in E^m_Z(U'')$ be a unique element such that $in^*(a'')=a$.
For an appropriate neighborhood $X''$ of the points
$x_j$'s ($j\in \{1,...,M\}$)
(containing all generic points of the scheme $X_v$,
but not containing any irreducible component of the scheme $X_v$)
and
the closure $Z''$ of $Z$ in $X''$ one can find an element
$$\tilde a''\in E^m_{Z''}(X'')$$
such that
$\tilde a''|_{U''}=a''\in E^m_Z(U'')$.

Applying Theorem \ref{An extended_moving_lemma} to the scheme
$X$, the points $x_1,x_2,\dots,x_M \in X$, the open $X''\subseteq X$, the closed subset $Z''\subset X''$, the ring $\mathcal O''$ and the scheme
$U''=Spec \mathcal~O''$,
we get a closed subset $Z^{new}$ in $X''$ containing $Z''$ with
$codim_{X''}(Z^{new})\geq 1$
and a morphism of pointed Nisnevich sheaves
$$\Phi_t: \Aff^1 \times U/(U-Z^{new}) \to X''/(X''-Z'')$$
in the category $Shv_{nis}(Sm/V)$. Moreover, \\
$$\Phi_0=[ U/(U-Z^{new})\xrightarrow{can} X/(X-Z^{new})\xrightarrow{p} X''/(X''-Z'')],$$
and $\Phi_1$ takes $ U/(U-Z^{new})$ to the distinguished point $*\in X''/(X''-Z'')$. \\
The
morphism $\Phi_t$ induces a map
$\Phi^*_t: E^m_{Z''}(X'')\to E^m_{\Aff^1\times Z^{new}}(\Aff^1\times U'')$
such that \\
(a) $i^*_1\circ \Phi^*_t: E^m_{Z''}(X'')\to E^m_{Z^{new}}(U'')$
is the zero map and \\
(b) $i^*_0\circ \Phi^*_t: E^m_{Z''}(X'')\to E^m_{Z^{new}}(U'')$
is the composite map of pull-backs \\
$E^m_{Z''}(X'')\to E^m_{Z}(U'')\to E^m_{Z^{new}\cap U''}(U'')$.\\\\
Since the pull-bach map
$E^m_{Z^{new}}(U'')\to E^m_{\Aff^1\times Z^{new}}(\Aff^1\times U'')$
is an isomorphism one has an equality
$i^*_1\circ \Phi^*_t=i^*_0\circ \Phi^*_t$.
Since $0=i^*_1\circ \Phi^*_t$ the map
$$E^m_{Z''}(X'')\to E^m_{Z}(U'')\xrightarrow{ext} E^m_{Z^{new}\cap U''}(U'')$$
vanishes.
Since $\tilde a''|_{U''}=a''\in E^m_Z(U'')$, hence $ext(a'')=0$ in $E^m_{Z^{new}\cap U''}(U'')$.
Recall that $Z$ is closed in $U''$. Consider a commutative diagram
\begin{equation}
\label{RactangelDiagram}
    \xymatrix{
a''\in E^m_{Z}(U'') \ar[rr]^{ext}\ar[d]_{in^*}&& E^m_{Z^{new}\cap U''}(U'')\ar[d]^{} &\\
a\in E^m_Z(U) \ar[rr]^{ext}&& E^m_{Z^{new}\cap U}(U), &\\ }
\end{equation}
in which the horizontal arrows are the support extension maps and the vertical ones are
the pullback maps (any support extension map is a pullback map).
Since $ext(a'')=0$ and $a=in^*(a'')$, hence $ext(a)=0$.
We proved that $ext_{2,1}=0$.

Let $c\geq 1$ and $m\in \mathbb Z$. Repeating literally the arguments above we prove the vanishing of the support extension map
$ext_{c+1,c}: E^m_{\geq c+1}(U)\to E^m_{\geq c}(U)$.
This gives for each integer $m$ a short exact sequence
$0\to E^m_{\geq c}(U)\to E^m_{(c)}(U)\xrightarrow{\partial} E^{m+1}_{\geq c+1}(U)\to 0$.

Prove now the assertion (b).
The vanishing of the map $ext_{2,1}$ yields the injectivity of the map pull-back map
$E^m_{\geq 1}(U)\to \oplus_{y\in U^{(1)}}E^{m+1}_y(U)$.
Consider the composite map
$$\partial: E^m(\eta)\xrightarrow{\partial_{\geq 1}} E^{m+1}_{\geq 1}(U)\to \oplus_{y\in U^{(1)}}E^{m+1}_y(U).$$
Since the second arrow in this composition is injective, hence
$\ker(\partial)=\ker(\partial_{\geq 1})$. But
$\ker(\partial_{\geq 1})=Im[E^m(U)\to E^m(\eta)].$
Thus, the sequence
$$E^m(U)\to E^m(\eta)\xrightarrow{\partial} \oplus_{y\in U^{(1)}}E^{m+1}_y(U)$$
is exact. The purity for $E^m$ is proved.
The assertion (a) is a direct consequence of statements (b),(c) and (d). The theorem is proved.
\end{proof}

\begin{proof}[Reducing Theorem \ref{cousin_1} to Theorem \ref{An extended_moving_lemma}]
Let $m\in \mathbb Z$ be as in Theorem \ref{cousin_1}. The hypotheses of the theorem yield
exactness of the complex
$0\to E^m(U)\to E^m(\eta)\xrightarrow{\partial} E^{m+1}_{\geq 1}(U)\to 0$.
This observation together with the item (d) of Theorem
\ref{General_Purity}
complete the proof of the theorem.
\end{proof}

\begin{proof}[Reducing Theorem \ref{General_Purity_z} to Theorem \ref{An extended_moving_lemma}]
The nearest aim is to prove the assertions (c) and (d).
So, take integers $m$ and $c$ with $c\geq 1$ and prove the vanishing of the support extension map
$$ext_{c+1,c}: E^m_{\geq c+1}(U_z)\to E^m_{\geq c}(U_z).$$
Take an $a\in E^m_{\geq c+1}(U)$. We may assume that $a\in E^m_Z(U)$ for a closed $Z$ in $U_z$ with
$codim_U(Z)\geq c+1$. For each closed point $x\in X$ which is in the closure $\bar {\{z\}}$ of $\{z\}$ in $X$
one has $U_z\subset U_x$. Write $\bar Z$ for the closure of $Z$ in $X$.
The continuity of the presheaf $E^m$ shows that there exist $x\in \bar {\{z\}}$
and an element $a_x\in E^m_{\bar Z\cap U_x}(U_x)$ such that $a_x|_{U_z}=a\in E^m_Z(U)$.
By the item (c) of Theorem \ref{General_Purity} there exists a closed subset $Z_{new}$ in $U_x$
containing $\bar Z\cap U_x$ and
of codimension $\geq c$ in $U_x$ such that $a_x$ vanishes in $E^m_{Z_{new}}(U_x)$. Thus,
$a$ vanishes in $E^m_{Z_{new}\cap U_z}(U_z)$. The assertion (c) is proved. Since the assertion (c)
yields the assertion (d) the assertion (d) is proved too.

The assertion (b) is an easy consequence of vanishing of the support extension map
$ext_{2,1}: E^m_{\geq c+1}(U_z)\to E^m_{\geq c}(U_z)$. Just repeat the arguments from the proof
of the item (b) of Theorem \ref{General_Purity_z}.

The assertion (a) is a direct consequence of statements (b),(c) and (d). The theorem is proved.
\end{proof}

\begin{proof}[Reducing Theorem \ref{cousin_1_z} to Theorem \ref{An extended_moving_lemma}]
Let $m\in \mathbb Z$ be as in Theorem \ref{cousin_1_z}. The hypotheses of the theorem yield
exactness of the complex
$0\to E^m(U_z)\to E^m(\eta)\xrightarrow{\partial} E^{m+1}_{\geq 1}(U_z)\to 0$.
This observation together with the item (d) of Theorem
\ref{General_Purity_z}
complete the proof of the first part of the theorem.

Now prove the second part. Let $i$ be an integer as in the second part. It is sufficient to show that the map
$E^i(U_z)\to E^i(\eta)$ is injective for $i=m$ and $i=m+1$. Let $\bar {\{z\}}$ be the closure of $\{z\}$ in $X$.
Then for each closed point $x\in \bar {\{z\}}$ one has the embedding $U_z\subseteq U_x$. Let $a\in E^i(U_z)$ be
an element vanishing in $E^i(\eta)$. By the continuity of the presheaf $E^i$ there exist a closed point
$x\in \bar {\{z\}}$ and an element $a_x\in E^i(U_x)$ such that $a_x|_{U_z}=a$. Since $a|_{\eta}=0$ it follows
that $a_x|_{\eta}=0$. Thus $a_x=0$ by the hypotheses of the second part of the theorem. Hence $a=0$.
The second part is proved.
\end{proof}

\section{Proof of Theorem \ref{cousin_BW}}\label{sect:cousin_BW}
\begin{proof}[Proof of Theorem \ref{cousin_BW}]
It is known \cite{B} that for each integer $m$ which is not divisible by $4$ the Balmer--Witt groups
$W^m(\mathcal O)$ and $W^m(\mathcal K)$ vanish. Particularly, for these integers $m$ the map
$\eta^*: W^m(\mathcal O)\to W^m(\mathcal K)$ is injective. Thus, by Theorem \ref{cousin_1}
the injectivity of the map $\eta^*: W(\mathcal O)\to W(\mathcal K)$
yields the exactness of
the Cousin complex
$Cous(W,U,m)$ for each integer $m$.
To prove the injectivity of the map $\eta^*: W(\mathcal O)\to W(\mathcal K)$ we closely follow ideas of \cite{OP2}.

Let $a\in W(\mathcal O)$ which vanishes in $W(\mathcal K)$.
The Mayer-Vietories property and the continuity of the presheaf $W|_{Sm'/V}$ on $Sm'/V$
allows to find appropriate closed points $x_{n+1}, ...,x_N$ in $X$,
 a Zariski neighborhood $X'$ of all the points $x_i$'s
($i=1,2,...,N$) as in the section \ref{agreements}
and an element $\tilde a\in W(X')$ such that
$a=\tilde a|_U$. Recall that
$X'$ contains
{\it all generic points}
of the scheme $X_v$,
but
{\it does not contain any irreducible component}
of the scheme $X_v$.

By the hypothesis of the theorem the element $\tilde a$ vanishes at the generic point $\eta$ of the scheme $X'$.
Let $\mathcal O_{X',X'_v}$ be the ring defined in
Notation \ref{inj_to_gen_point}. It is a semi-local Dedekind domain.
Thus, the homomorphism
$W(\mathcal O_{X',X'_v})\to W(\mathcal K)$ is injective. Hence $\tilde a$ vanishes in $W(\mathcal O_{X',X'_v})$.
So, there is a non-zero $f\in \Gamma(X',\mathcal O_{X'})$ such that
the closed subset $Y:=\{f=0\}$ of $X'$
does not contain any irreducible component of the scheme $X'_v$.
and the element $\tilde a|_{X'_f}$ vanishes.

Applying Corollary \ref{Homotopy_w_Traces} to the scheme
$X$, the points $x_1,x_2,\dots,x_N \in X$, the open $X'\subseteq X$á the closed subset $Y\subset X'$
we get an affine Zariski open neighborhood $U'$ of points
$x_1,x_2,\dots,x_N \in X'$ and
a diagram of $A$-schemes and $A$-morphisms of the form
$$\Aff^1_{U'} \xleftarrow{\tau} \mathcal X \xrightarrow{p_X} X'$$
with an $A$-smooth irreducible scheme $\mathcal X$ and a finite surjective morphism $\tau_{new}$
such that the canonical sheaf $\omega_{\mathcal X/V}$ is isomorphic to the structure sheaf $\mathcal O_{\mathcal X}$.
Moreover, if we write $p: \mathcal X\to U'$ for the composite map $pr_{U'}\circ \tau$
and $\mathcal Y$ for $p^{-1}_X(Y)$, then these data
enjoy the following properties:\\
(0) the closed subset $\mathcal Y$ of $\mathcal X$ is quasi-finite over $U'$;\\
(1) there is a section $\Delta_{}: U'\to \mathcal X$ of the morphism $p_{}$ such that $\tau_{}\circ \Delta_{}=i_0$ and $p_X\circ \Delta_{}=can$,
where $i_0$ is the zero section of $\Aff^1_{U'}$ and $can: U'\hookrightarrow X'$ is the inclusion;\\
(2) the morphism $\tau_{}$ is \'{e}tale as over $\{0\}\times U'$, so over $\{1\}\times U'$;\\
(3) for $\mathcal D_1:=\tau_{}^{-1}(\{1\}\times U')$ and $\mathcal Y:=p^{-1}_X(Y)$ one has $\mathcal D_1\cap \mathcal Y=\emptyset$; \\
(4) for $\mathcal D_0:=\tau_{}^{-1}(\{0\}\times U')$ one has $\mathcal D_0=\Delta_{}(U')\sqcup \mathcal D'_0$ and $\mathcal D'_0\cap \mathcal Y=\emptyset$;\\

Consider a category $\text{Aff}$ of affine $\Aff^1_{U'}$-schemes and $\Aff^1_{U'}$-morphisms.
For a scheme $T\in \text{Aff}$ write $\mathcal T$ for $T\times_{\Aff^1_{U'}} \mathcal X$.
There are two interesting presheaves on $\text{Aff}$:
$$T\mapsto W(\mathcal T) \ \text{and} \ T\mapsto W(T).$$
Choose an isomorphism
$l: \mathcal O_{\mathcal X}\to \omega_{\mathcal X/V}$
of the $\mathcal O_{\mathcal X}$-modules. Similarly to \cite[Section 12]{OP2} this isomorphism defines
a functor transformation (an Euler trace)
$$Tr^{\mathcal E}: T\mapsto [Tr^{\mathcal E}_{\mathcal T/T}:  W(\mathcal T)\to W(T)],$$
which enjoy properties similar to the properties (1) to (6) as in \cite[Section 12]{OP2}.
Set $\alpha=p^*_X(\tilde a)\in W(\mathcal X)$. Then following \cite[Section 3]{OP2} one gets an equality in $W(U')$
$$Tr^{\mathcal E}_{\mathcal D_1/U'}(\alpha|_{\mathcal D_1})=Tr^{\mathcal E}_{\mathcal D'_0/U'}(\alpha|_{\mathcal D'_0})+
u\cdot \Delta_{}^*(\alpha)$$
for a unit
$u\in \Gamma(U', \mathcal O^{\times}_{})$. Our choice of the element $f$ and the properties (3),(4) show that
$\alpha|_{\mathcal D_1}=0$ and $\alpha|_{\mathcal D'_0}=0$. Hence $\Delta_{}^*(\alpha)=0$.
By the property (1) one has $p_X\circ \Delta_{}=can$.
Thus, $\tilde a|_{U'}=\Delta_{}^*(\alpha)=0$. Hence $0=(\tilde a|_{U'})|_U=a$ in $W(U)$.
The injectivity of the map $W(\mathcal O)\to W(\mathcal K)$ is proved.
Theorem \ref{cousin_BW} is proved.
\end{proof}

\section{Presheaves with transfers}
Let $A$, $p>0$, $V$, $Sm'/V$, $Sch'/V$ be as in Section \ref{agreements}.

\begin{defn}\label{Pre_w_Tr}
A presheaf with transfers on $Sch'/V$ is an additive presheaf \ $\cal F$ such that for each finite flat
morphism $\pi: Y\to S$ in $Sch'/V$ there is given a homomorphism $\Tr_{Y/S}: \mathcal F(Y)\to \mathcal F(S)$.
Moreover these homomorphisms enjoy the following properties \\
(i) base change: if $f: S'\to S$ is a $V$-morphism, then for $Y'=S'\times_S Y$ and $F=pr_Y: Y'\to Y$
one has an equality $f^*\circ \Tr_{Y/S}=\Tr_{Y'/S'}\circ F^*: \mathcal F(Y)\to \mathcal F(S')$;\\
(ii) additivity: if $Y=Y'\sqcup Y''$, then $Tr_{Y/S}=Tr_{Y'/S}+Tr_{Y''/S}$;\\
(iii) normalization: if $\pi=id_S: S\to S$, then $Tr_{Y/S}=id$.

An assignment $[\pi: Y\to S]\mapsto [Tr_{Y/S}: \mathcal F(Y)\to \mathcal F(S)]$
subjecting conditions (1)--(3) is called often a trace structure on the presheaf $\mathcal F$.
So, a presheaf with transfers on $Sch'/V$ is an additive presheaf $\cal F$ equipped with a trace structure.

An \'{e}tale sheaf with transfers on $Sch'/V$ is a presheaf with transfers which is an  \'{e}tale sheaf
on $(Sch'/V)_{Et}$.
\end{defn}

\begin{exam}\label{Tr_on_Et_Coh}
Let $\mathcal F$ be an \'{e}tale sheaf with transfers on $Sch'/V$ and let $m\geq 0$ be an integer.
Let $\pi: Y\to S$ in $Sch'/V$ be a finite flat
morphism.
Let $S_{et}$ be the small \'{e}tale site of $S$.
Each $S'\in S_{et}$ is in $Sch'/V$ and each morphism $S''\to S'$ in $S_{et}$ is a morphism in $Sch'/V$.
And similar observation is true for the small \'{e}tale site $Y_{et}$.
Then on the small \'{e}tale site $S_{et}$ there are two \'{e}tale sheaves:
$\mathcal F$ and $\pi_*(\mathcal F)$.
Since
$\mathcal F$ is an \'{e}tale sheaf with transfers on $Sch'/V$ it follows there is given a sheaf morphism
$\Tr_{Y/S}: \pi_*(\mathcal F)\to \mathcal F$ on $S_{et}$. This sheaf morphism induces a group homomorphism
$$H^{m}_{et}(Y,\mathcal F)=H^{m}_{et}(S,\pi_*(\mathcal F))\xrightarrow{Tr^m_{Y/S}} H^{m}_{et}(S,\mathcal F).$$
Clearly, the additive presheaf $H^m_{et}(-,\mathcal F)$ on $Sch'/V$ together with the homomorphisms $Tr^m_{Y/S}$
is a presheaf with transfers on $Sch'/V$. The trace structure $Y/S \mapsto Tr^m_{Y/S}$ on the
presheaf $H^m_{et}(-,\mathcal F)$ is called
{\it the induced trace structure}.
\end{exam}

The following two examples show that
{\it constant and locally constant \'{e}tale sheaves
} on $Sch'/V$
are naturally sheaves with transfers. And therefore their cohomology presheaves
are equipped with distinguished trace structures.
\begin{exam}\label{Const_is_w_Tr}
Let $A$ be an abelian group and let $d$ be an integer. Write $m_d: A\to A$ for the multiplication by $d$.
Let $\underline A$ be the constant \'{e}tale sheaf on $(Sch'/V)_{et}$.
We show now that the sheaf $\underline A$ is naturally a sheaf with transfers on $Sch'/V$.
Let $\pi: Y\to S$ be a surjective finite flat morphism in $Sch'/V$. Suppose schemes $S$ and $Y$ are connected. Then
$\pi_*(\mathcal O_Y)$ is a locally free $\mathcal O_S$ module. Since $S$ is connected this module
has a constant rank, say $[Y:S]$. In this case put $Tr_{Y/S}=m_{[Y:S]}: A=\underline A(Y)\to \underline A(S)=A$.

Suppose $S$ is connected and $Y=\sqcup_{i\in I}Y_i$, where each $Y_i$ is connected. Then each $Y_i$ is surjective finite flat over $S$.
Put $Tr_{Y/S}=\Sigma_{i\in I}Tr_{Y_i/S}: \oplus_{i\in I}\underline A(Y_i)\to \underline A(S)$. Finally, let
$S=\sqcup_{j\in J}S_j$, where each $S_j$ is connected. Put $Y_j=\pi^{-1}(S_j)$. Since $\pi$ is surjective it follows
that $Y_j$ is not empty. Put
$$Tr_{Y/S}=\oplus_{j\in J}Tr_{Y_j/S_j}: \oplus_{j\in J}\underline A(Y_j)\to \oplus_{j\in J}\underline A(S_j).$$
It is easy to check that the assignment
$[\pi: Y\to S]\mapsto [Tr_{Y/S}: \underline A(Y)\to \underline A(S)]$
defines a trace structure on the constant sheaf $\underline A$. This trace structure is called
{\it the standard trace structure} on $\underline A$.

By Example \ref{Tr_on_Et_Coh} for each integer $m$ the additive presheaf $H^m_{et}(-,\underline A)$ on $Sch'/V$ is equipped
with the trace structure induced the standard trace structure on $\underline A$. It is called
{\it the standard trace structure} on $H^m_{et}(-,\underline A)$.
\end{exam}

\begin{exam}\label{Loc_Const_is_w_Tr}
Recall that an \'{e}tale sheaf $\mathcal F$ on $Sch'/V$ is called locally constant if there exists an abelian group $A$
such that the sheaf $\mathcal F$ is isomorphic to the constant sheaf $\underline A$ locally for the \'{e}tale topology on
$Sch'/V$.

We prove now that the sheaf $\mathcal F$ is naturally a sheaf with transfers on $Sch'/V$.
For each $S\in Sch'/V$ and each point $x\in S$ denote $\mathcal F_x$ the stalk of $\mathcal F$ at the point $x$.
That is $\mathcal F_x=\mathcal F(S^{sh}_x)$, where $S^{sh}_x$ is the strict henselization of $S$ at the point $x$.
First note that for any surjective finite flat morphism $\pi: Y\to S$ and each point $x\in S$ and each point
$y\in Y$ over $x$ the pull-back map $\mathcal F_x\to \mathcal F_y$ is an isomorphism. This yields
that $\pi_*(\mathcal F)_x=\oplus_{y/x}\mathcal F_y=\oplus_{y/x}\mathcal F_x$.
Put $Tr_{\pi/x}=\Sigma_{y/x}m_{[Y^{sh}_y:S^{sh}_x]}: \oplus_{y/x}\mathcal F_x \to \mathcal F_x$.

The following claim we left to the reader:
there exists a unique trace structure
$$[\pi: Y\to S]\mapsto [Tr_{Y/S}: \pi_*(\mathcal F)(S)=\mathcal F(Y)\to \mathcal F(S)]$$
on $\mathcal F$ such that for each $S\in Sch'/V$ and each point $x\in S$ one has an equality
$$(Tr_{Y/S})_x=Tr_{\pi/x}: \pi_*(\mathcal F)_x\to \mathcal F_x.$$
This trace structure is called
{\it the standard trace structure} on $\mathcal F$.

By Example \ref{Tr_on_Et_Coh} for each integer $m$ the additive presheaf $H^m_{et}(-,\mathcal F)$ on $Sch'/V$ is equipped
with the trace structure induced by the standard trace structure on $\mathcal F$. It is called
{\it the standard trace structure} on $H^m_{et}(-,\mathcal F)$.
\end{exam}

\begin{exam}\label{Tr_on_TT_K_th}
Let $m$ be an integer.
Clearly, the Thomason-Throughbor $K_m$-groups is a presheaf with transfers on $Sch'/V$.
The Thomason-Throughbor $K_m$-groups with finite coefficients $\mathbb Z/r$ is also a presheaf with transfers on $Sch'/V$.
If $\mathcal A$ be an Azumaya algebra over $A$, then the Thomason-Throughbor $K_m$-groups with coefficients in $\mathcal A$
is also a presheaf with transfers on $Sch'/V$.
\end{exam}



\begin{defn}\label{dim_one_property}
Let $\mathcal F$ be a presheaf of abelian groups on $Sch'/V$. One says that $\mathcal F|_{Sm'/V}$ satisfies
$\text{dim 1 property}$, if for each data
$d\geq 1$, $X$, $X'$, $\eta$
as in Section \ref{agreements} and for $\mathcal S=\mathcal S'$ as in Notation \ref{inj_to_gen_point}
the map $\eta^*: \mathcal F(\mathcal S)=\mathcal F(\mathcal S')\to \mathcal F(\eta)$
is injective.
\end{defn}

\begin{defn}\label{coh_type}
Let $\mathcal F$ be a presheaf on $Sch'/V$. One says that $\mathcal F|_{Sm'/V}$
is of cohomological type \\
a) if for each scheme $S\in Sm'/V$
the map $p^*_{S}: \mathcal F(S)\to \mathcal F(\mathbb A^1_S)$ is an isomorphism;\\
b) if for each $S\in Sm'/V$ and each Zariski open cover $S_1\cup S_2=S$ the standard sequence is exact
$$\mathcal F(S)\to \mathcal F(S_1)\oplus \mathcal F(S_2)\to \mathcal F(S_1\cap S_2).$$
\end{defn}

\begin{defn}\label{cont_presh}(Continuity property)
One says that a presheaf $G$ on $Sch'/V$ is continuous
if for each $S\in Sch'/V$, each left filtering system $(S_j)_{j\in J}$
as in section \ref{agreements}
and each $V$-scheme isomorphism $S\to \lim_{j\in J} S_j$
the map
$$\text{colim}_{j\in J}G(S_j)\to G(S)$$ is an isomorphism.
Clearly, if a presheaf $G$ on $Sch'/V$ is continuous, then
$G|_{Sm'/V}$ is continuous.
\end{defn}



\begin{thm}\label{inj_main}{\rm
Let $\mathcal F$ be an additive continuous presheaf with transfers on $Sch'/V$. Suppose
$\mathcal F|_{Sm'/V}$ satisfies the $\text{dim 1 property}$
and is of cohomological type.
Then for each data $d\geq 1$, $X$, $x_1,x_2,\dots,x_n\in X$, $\mathcal O$, $U$, $\eta$
as in Section \ref{agreements} the map
$\eta^*: \mathcal F(U)\to \mathcal F(\eta)$ is injective.
}
\end{thm}

\begin{proof}[Proof of Theorem \ref{inj_main}]
Let $a\in \mathcal F(U)$ which vanishes in $\mathcal F(\eta)$.
The Mayer-Vietories property and the continuity of the presheaf $\mathcal F|_{Sm'/V}$ on $Sm'/V$
allows to find appropriate closed points $x_{n+1}, ...,x_N$ in $X$,
 a Zariski neighborhood $X'$ of all the points $x_i$'s
($i=1,2,...,N$) as in the section \ref{agreements}
and an element $\tilde a\in \mathcal F(X')$ such that
$a=\tilde a|_U$. Recall that
$X'$ contains
{\it all generic points}
of the scheme $X_v$,
but
{\it does not contain any irreducible component}
of the scheme $X_v$.


By the hypothesis of the theorem the element $\tilde a$ vanishes at the generic point $\eta$ of the scheme $X'$.
Since $\mathcal F$ is continuous and
$\mathcal F|_{Sm'/V}$ satisfies the
$\text{dim 1 property}$
it follows there is a non-zero $f\in \Gamma(X',\mathcal O_{X'})$ such that
the closed subset $Y:=\{f=0\}$ of $X'$
does not contain any irreducible component of the scheme $X'_v$
and the element $\tilde a|_{X'_f}$ vanishes.

Applying Corollary \ref{Homotopy_w_Traces} to the scheme
$X$, the points $x_1,x_2,\dots,x_N \in X$, the open $X'\subseteq X$ and the closed subset $Y\subset X'$,
we get
an affine Zariski open neighborhood $U'$ of points
$x_1,x_2,\dots,x_N \in X'$ and
a diagram of $A$-schemes and $A$-morphisms of the form
$$\Aff^1_{U'} \xleftarrow{\tau} \mathcal X \xrightarrow{p_X} X'$$
with an $A$-smooth irreducible schemes $\mathcal X$ and a finite surjective morphism $\tau$.
Moreover, if we write $p: \mathcal X\to U'$ for the composite map $pr_{U'}\circ \tau$, then these data
enjoy the following properties:\\
(1) there is a section $\Delta: U'\to \mathcal X$ of the morphism $p$ such that $\tau\circ \Delta=i_0$ and $p_X\circ \Delta=can$,
where $i_0$ is the zero section of $\Aff^1_{U'}$ and $can: U'\hookrightarrow X'$ is the canonical inclusion;\\
(2) the morphism $\tau$ is \'{e}tale as over $\{0\}\times U'$, so over $\{1\}\times U'$;\\
(3) for $\mathcal D_1:=\tau^{-1}(\{1\}\times U')$ and $\mathcal Y:=p^{-1}_X(Y)$ one has $\mathcal D_1\cap \mathcal Y=\emptyset$; \\
(4) for $\mathcal D_0:=\tau^{-1}(\{0\}\times U')$ one has $\mathcal D_0=\Delta(U')\sqcup \mathcal D'_0$ and $\mathcal D'_0\cap \mathcal Y=\emptyset$;\\

Consider a category $\text{Aff}$ of affine $\Aff^1_{U'}$-schemes and $\Aff^1_{U'}$-morphisms.
For a scheme $T\in \text{Aff}$ write $\mathcal T$ for $T\times_{\Aff^1_{U'}} \mathcal X$.
There are two interesting presheaves on $\text{Aff}$:
$$T\mapsto \mathcal F(\mathcal T) \ \text{and} \ T\mapsto \mathcal F(T).$$
The trace structure on the presheaf $\mathcal F$ defines a functor transformation
$$Tr: T\mapsto [Tr_{\mathcal T/T}:  \mathcal F(\mathcal T)\to \mathcal F(T)],$$
which enjoys properties (ii) and (iii) as in Definition \ref{Pre_w_Tr}.
Set $\alpha=p^*_X(\tilde a)\in \mathcal F(\mathcal X)$.
Using now the properties (2) and (3) and the homotopy invariance of the presheaf
$\mathcal F|_{Sm'/V}$
one gets an equality in $\mathcal F(U')$
$$Tr_{\mathcal D_1/U'}(\alpha|_{\mathcal D_1})=Tr_{\mathcal D'_0/U'}(\alpha|_{\mathcal D'_0})+
\Delta^*(\alpha).$$
Our choice of the element $f$ and the properties (3),(4) show that
$\alpha|_{\mathcal D_1}=0$ and $\alpha|_{\mathcal D'_0}=0$. Hence $\Delta^*(\alpha)=0$.
By the property (1) one has $p_X\circ \Delta=can$.
Thus, $\tilde a|_{U'}=\Delta^*(\alpha)=0$. Hence $0=(\tilde a|_{U'})|_U=\tilde a|_U=a$ in $\mathcal F(U)$.
The injectivity of the map $\eta^*: \mathcal F(U)\to \mathcal F(\eta)$ is proved.
Theorem \ref{inj_main} is proved.
\end{proof}

\section{Proof of Theorems \ref{cousin_K_w_Z_rZ}, \ref{cousin_H_et_loc_const_F}, \ref{cousin_H_et_mu} and \ref{cousin_K2}.}
Here are some examples of presheaves on $Sch'/V$ which are continuous.
\begin{exam}\label{cont_of_K_th}
For each integer $m$ the presheaves $K^m(-)$ and $K^m(-,\mathbb Z/r)$ on $Sch'/V$ satisfy the continuity property
as in Definition \ref{cont_presh}.
\end{exam}

\begin{exam}\label{cont_of_et_coh}
Let $\cal F$ be a locally constant $r$-torsion
sheaf on the big \'{e}tale site $(Sch'/V)_{et}$.
Then for each integer $m$ the presheaf $H^m_{et}(-,\cal F)$ on $Sch'/V$ satisfies the the continuity property.
\end{exam}

\begin{lem}\label{dim_one_K}
Let $\mathcal S$ be semi-local irreducible Dedekind scheme $s_1,...,s_l\in \mathcal S$ be all its closed points.
Suppose that all the residue fields $k(s_j)$ have characteristic $p$ and the general point
$\eta:=\mathcal S-\{s_1,...,s_l\}$ has characteristic zero.
Let $r$ be coprime to $p$.
Then for each integer $m$ the map $K_m(\mathcal S,\mathbb Z/r)\to K_m(\eta,\mathbb Z/r)$ is injective.
\end{lem}
\begin{proof}(of the lemma) Let $i_j: w_j\hra W$ is the closed embedding.
It is sufficient to check that for each $j=1,...,l$ the direct image map
$(i_j)_*: K_m(w_j,\mathbb Z/r)\to K_m(W,\mathbb Z/r)$ is zero.
We may assume that $j=1$ and write $w$ for $w_1$ and $i$ for $i_1$.
So, it is sufficient to check that the map
$i_*: K_m(w,\mathbb Z/r)\to K_m(W,\mathbb Z/r)$ is zero.
Let $a\in K_m(w,\mathbb Z/r)$. We must check that $i_*(a)=0$ in $K_m(W,\mathbb Z/r)$.
By the Suslin rigidity theorem there is an etale neibourhood
$(W\xleftarrow{Pi^{\circ}} W^{\circ}, s: w\to W^{\circ})$ of the point $w$
and an element $a^{\circ} \in K_m(W^{\circ},\mathbb Z/r)$ such that $s^*(a^{\circ})=a$. Write
$\Pi^{\circ}$ as $\Pi\circ in$, where $in: W^{\circ}\hra \tilde W$ is an open embedding, $\Pi: \tilde W\to W$
is a finite morphism and $\tilde W$ is a semi-local irreducible Dedekind scheme.
Let $\tilde \eta\in \tilde W$ be the general point of $\tilde W$.
Then $\tilde \eta=\tilde W-\Pi^{-1}(\underline w)$, where $\underline w=\{w_1,...,w_l\}$.
Clearly, $\tilde \eta$ is in $W^{\circ}$. Put $\tilde a:=a^{\circ}|_{\tilde \eta}$.
Let $\partial: K_{m+1}(\tilde \eta,\mathbb Z/r)\to K_m(\Pi^{-1}(\underline w);\mathbb Z/r)$
be the boundary map.

Choose a function $f\in \Gamma(\tilde W-\{s(w)\},\mathcal O^{\times})$ such that
$f|_{\Pi^{-1}(\underline w)-s(w)}\equiv 1$ and $div_{s(w)}(f)=1$. In this case
$\partial_{s(w)}(f\cup \tilde a)=a\in K_m(s(w),\mathbb Z/r)$ and
for each $\tilde w\in \Pi^{-1}(\underline w)-s(w)$ one has
$\partial_{\tilde w}(f\cup \tilde a)=0 \in K_m(\tilde w,\mathbb Z/r)$.

Let $\underline i: \underline w \hra W$ and $\tilde {\underline i}: \Pi^{-1}(\underline w)\hra \tilde W$ be the closed embeddings and
$\pi=\Pi|_{\Pi^{-1}(\underline w)}: \Pi^{-1}(\underline w)\to \underline w$.
The above computation with the element
$f\cup \tilde a\in K_{m+1}(\tilde \eta;\mathbb Z/r)$
shows that
$(\underline i\circ \pi)_*(\partial (\tilde a))=i_*(a)$.
Since $\Pi\circ \tilde {\underline i}=\underline i\circ \pi: \Pi^{-1}(\underline w)\to W$
and $\tilde {\underline i}_*\circ \partial=0$ it follows that
$0=i_*(a)$. This proves the lemma.
\end{proof}

\begin{proof}[Proof of Theorem \ref{cousin_K_w_Z_rZ}]
Let $m$ be an integer.
By Examples \ref{Tr_on_TT_K_th} and \ref{cont_of_K_th} the presheaf $K^m(-;\mathbb Z/r)$
is an additive continuous presheaf with transfers on $Sch'/V$.
By Lemma \ref{dim_one_K} the presheaf $K^m(-;\mathbb Z/r)|_{Sm'/V}$ on $Sm'/V$
satisfies the $\text{dim 1 property}$.
The Thomason-Throughbor $K$-theory
$(K^*(-;\mathbb Z/r),\partial)$
with finite coefficients $\mathbb Z/r$
is a cohomology theory on $Sm'Op/V$ as explained just above
Theorem \ref{cousin_K_w_Z_rZ}.
Thus, the presheaf $K^m(-;\mathbb Z/r)|_{Sm'/V}$ is of cohomological type
in the sense of Definition \ref{coh_type}.
Thus, by Theorem \ref{inj_main}
the map
$\eta^*: K^m(U;\mathbb Z/r)\to K^m(\eta;\mathbb Z/r)$
is injective. Now Theorem
\ref{cousin_1} completes the proof.
\end{proof}

\begin{lem}\cite{CHK}\label{dim_one_Et_Coh}
Let $\cal F$ be a a locally constant $r$-torsion
sheaf on the big \'{e}tale site \'{E}t/V.
Let $\mathcal S$, $s_1,...,s_l\in \mathcal S$, $\eta$ be such as in Lemma \ref{dim_one_K}.
Then for each integer $m$ the map $H^m_{et}(\mathcal S,\mathcal {F})\to H^m_{et}(\eta,\cal F)$ is injective.
\end{lem}

\begin{proof}[Proof of Theorem \ref{cousin_H_et_loc_const_F}]
Let $m$ be an integer. By Examples \ref{Loc_Const_is_w_Tr} and \ref{cont_of_et_coh}
the presheaf $H^m_{et}(-,\mathcal F)$
is an additive continuous presheaf with transfers on $Sch'/V$.
By Lemma \ref{dim_one_Et_Coh} the presheaf $H^m_{et}(-,\mathcal F)|_{Sm'/V}$ on $Sm'/V$
satisfies the $\text{dim 1 property}$.
The \'{e}tale cohomology theory $(H^*_{et}(-, \mathcal F),\partial)$ on $Sm'Op/V$
is a cohomology theory on $Sm'Op/V$. Thus, the presheaf $H^m_{et}(-,\mathcal F)$
is of cohomological type in the sense of Definition \ref{coh_type}.
Thus, by Theorem \ref{inj_main}
the map
$\eta^*: H^m_{et}(U,\mathcal F)\to H^m_{et}(\eta,\mathcal F)$ is injective.
Theorem
\ref{cousin_1} completes the proof.
\end{proof}

\begin{proof}[Proof of Theorem \ref{cousin_H_et_mu}]
This theorem is just a particular case of Theorem
\ref{cousin_H_et_loc_const_F}.
\end{proof}

\begin{lem}\label{dim_one_K_th}
Let $W$ be semi-local irreducible Dedekind scheme. Let $\eta\in W$ be its general point.
Then the maps $K_2(W)\to K_2(\eta)$ is injective.
\end{lem}

\begin{proof}
Let $w_1,...,w_l\in W$ are all its closed points and $\underbar w=\{w_1,...,w_l\}$
be the closed subscheme of $W$ and $i: \underbar w\hra W$ be the closed embedding.
It is sufficient to prove that the boundary map
$\partial: K_3(\eta)\to K_2(\underbar w)$ is surjective.

Let $a\in K_2(\underbar w)$ be an element. It is sufficient to find an element
$\alpha \in K_3(\eta)$ such that $\partial(\alpha)=a$. Recall that
by Matzumoto theorem the pull-back map $i^*: K_2(W)\to K_2(\underbar w)$ is surjective.
Thus, there exists an $\tilde a\in K_2(W)$ with $i^*(\tilde a)=a$.
There is a rational function $f$ on $W$ such that for each point $w\in \underbar w$ one has $div_{w}(f)=-1$.
Take $\tilde a\cup (f)\in K_3(\eta)$ as $\alpha$. Then $\partial(\alpha)=a$.
\end{proof}

\begin{proof}[Proof of Theorem \ref{cousin_K2}]
Let $m$ be an integer.
By Example \ref{Tr_on_TT_K_th} and \ref{cont_of_K_th} the presheaf $K^{-2}(-)$ is
an additive continuous presheaf with transfers on
$Sch'/V$.
By Lemma \ref{dim_one_K_th} the presheaf $K^{-2}(-)|_{Sm'/V}$ on $Sm'/V$
satisfies the $\text{dim 1 property}$.
The Thomason-Throughbor $K$-theory
$(K^*(-),\partial)$
is a cohomology theory on $Sm'Op/V$ as explained just above
Theorem \ref{cousin_K_w_Z_rZ}.
Thus, the presheaf $K^{-2}(-)$ is of cohomological type
in the sense of Definition \ref{coh_type}.
Thus, by Theorem \ref{inj_main}
the map $\eta^*: K_2(U)\to K_2(\eta)$ is injective.
The map $\eta^*: K_1(U)\to K_1(\eta)$ is injective also.
Now Theorem
\ref{cousin_1} proves exactness of the complex
\eqref{thm:Gersten_is_Exact}.

It remains to prove the exactness of the complex
\eqref{thm:Gersten_mod_n_is_Exact}.
To do this recall that for each scheme $S\in Sch'/V$ and each integer $n$ there is the coefficient short exact sequence
$0\to K_n(S)/r \to K_n(S,\mathbb Z/r)\to {_{r}}K_{n-1}(S)\to 0$. For an integer $n$ let
$(\ref{thm:Gersten_w_Z/r})_n$ be the complex \eqref{thm:Gersten_w_Z/r} with $m=n$.
Consider a complex
\begin{equation}\label{mu_and_mu}
0\to {_{r}}K_1(U)\cong {_{r}}K_1(\eta) \to 0\to 0\to 0
\end{equation}
and the coefficient short exact sequence
$0\to (\ref{thm:Gersten_mod_n_is_Exact}) \to (\ref{thm:Gersten_w_Z/r})_2 \to \text{\eqref{mu_and_mu}} \to 0$
of complexes. By Theorem \ref{cousin_K_w_Z_rZ} the complex $(\ref{thm:Gersten_w_Z/r})_2$ is exact. Clearly,
the complex \eqref{mu_and_mu} is exact. Thus, the complex \eqref{thm:Gersten_mod_n_is_Exact} is also exact.
The theorem is proved.
\end{proof}

\section{Suslin's exact sequence in mixed characteristic}
{\it Our next aim is to get a version of Suslin's short exact sequence in mixed characteristic}.
Let $A$, $p>0$, $d\geq 1$, $X$, $x_1,x_2,\dots,x_n\in X$, $\mathcal O$, $U$, $\mathcal K$  be as in Section \ref{agreements}.
Let $r$ be an integer coprime to $p$.
\begin{proof}[Proof of Theorem \ref{thm:Suslin_exact}]
Consider a commutative diagram
\begin{equation}
\label{Big_Diagram}
    \xymatrix{
0 \ar[r] & _{r}K_2(\mathcal K) \ar[rr]^{\partial} \ar[d]_{}&& \oplus_{x\in X^{(1)}}\mu_r(k(x)) \ar[d] \ar[r] & 0 \ar[d] \ar[r] & 0  \\
0 \ar[r] & K_2(\mathcal K) \ar[rr]^{\partial} \ar[d]_{\times r} && \oplus_{x\in X^{(1)}} K_1(k(x)) \ar[r]^{\partial} \ar[d]_{\times r} & \oplus_{y\in X^{(2)}}K_0(k(y)) \ar[d]_{\times r}
\ar[r] & 0 \\
0 \ar[r] & K_2(\mathcal K) \ar[rr]^{\partial} \ar[d]_{\alpha} && \oplus_{x\in X^{(1)}} K_1(k(x)) \ar[r]^{\partial} \ar[d]_{\alpha} & \oplus_{y\in X^{(2)}}K_0(k(y)) \ar[d]_{\alpha}
\ar[r] & 0 \\
0 \ar[r] & K_2(\mathcal K)/r \ar[rr]^{\partial} \ar[d]_{} && \oplus_{x\in X^{(1)}} K_1(k(x))/r \ar[r]^{\partial} \ar[d]_{} & \oplus_{y\in X^{(2)}}K_0(k(y))/r \ar[d]_{}
\ar[r] & 0 \\
  & 0                                     &&   0                                                   & 0, \\
}
\end{equation}
where $\alpha$ are the reduction modulo $r$ homomorphisms.
Theorem \ref{cousin_K2} shows that the second, the third and the forth horizontal complexes
compute Zarisky cohomology $H^*_{Zar}(X,\underline K_2)$, $H^*_{Zar}(X,\underline K_2)$ and
$H^*_{Zar}(X,\underline K_2/r)$ respectively.
The morphism $\alpha$ between the third and the forth row of the diagram induces
a homomorphism
$\alpha: H^1_{Zar}(X,\underline K_2)/r \to H^1_{Zar}(X,\underline K_2/r)$.
Define a morphism
$\beta: H^1_{Zar}(X,\underline K_2/r) \to {_{r}}H^2_{Zar}(X,\underline K_2)$
as follows: for an element
$a\in \oplus_{x\in X^{(1)}} K_1(k(x))/r$ with $\partial(a)=0$
choose an element
$\tilde a\in \oplus_{x\in X^{(1)}} K_1(k(x))$
such that $\alpha (\tilde a)=a$.
Then $\alpha (\partial (\tilde a))=0$.
Thus there exists a unique element $b\in \oplus_{y\in X^{(2)}}K_0(k(y))$.
Let $\bar b$ be the class of $b$ modulo $Im(\partial)$.
Put $\beta(a)=\bar b$.
Clearly, $\bar b$ is in ${_{r}}H^2_{Zar}(X,\underline K_2)$.
Also, $\beta\circ \alpha=0$. Thus, we get a complex of the form
\begin{equation}
\label{Suslin_sequence}
0\to H^1_{Zar}(X,\underline K_2)/r \xrightarrow{\alpha} H^1_{Zar}(X,\underline K_2/r) \xrightarrow{\beta} {_{r}}H^2_{Zar}(X,\underline K_2)\to 0,
\end{equation}

\begin{lem}\label{prop:Suslin_exact}
The sequence \eqref{Suslin_sequence} is short exact.
\end{lem}

\begin{proof}
Use a diagram chase in the diagram \eqref{Big_Diagram}.
\end{proof}
Recall that $NH^3_{et}(X,\mu^{\otimes 2}_n)=ker[H^3_{et}(X,\mu^{\otimes 2}_n)\xrightarrow{\eta^*} H^3_{et}(\eta,\mu^{\otimes 2}_n)]$.
{\it The nearest aim is to identify
$H^1_{Zar}(X,\underline K_2/n)$
with
$NH^3_{et}(X,\mu^{\otimes 2}_n)$.
}
Recall that the norm residue homomorphisms identify the forth row of the diagram
\eqref{Big_Diagram} with the complex
\begin{equation}\label{Bloch_Ogus_X}
0\to H^2(\mathcal K,\mu^{\otimes 2}_n) \xrightarrow{\partial} \oplus_{x\in X^{(1)}}H^1(k(x),\mu_n) \xrightarrow{\partial}
\oplus_{y\in X^{(2)}}\mathbb Z/n \to 0
\end{equation}
By Theorem \ref{cousin_H_et_mu} the first cohomology of the complex \eqref{Bloch_Ogus_X} equals to
$H^1_{Zar}(X,\mathcal H^2)$. Thus, $H^1_{Zar}(X,\underline K_2/n)=H^1_{Zar}(X,\mathcal H^2)$.
To identify
$H^1_{Zar}(X,\underline K_2/n)$
with
$NH^3_{et}(X,\mu^{\otimes 2}_n)$
it is remained to check that
$H^1_{Zar}(X,\mathcal H^2)=NH^3_{et}(X,\mu^{\otimes 2}_n)$.

Theorem \ref{cousin_H_et_mu} yields a spectral sequence of the form
$H^p_{Zar}(X,\mathcal H^q)\Rightarrow H^{p+q}_{et}(X,\mu^{\otimes 2}_n)$.
The latter means that the group $H^{3}_{et}(X,\mu^{\otimes 2}_n)$
relates to the groups
$H^0_{Zar}(X,\mathcal H^3)$, $H^1_{Zar}(X,\mathcal H^2)$, $H^2_{Zar}(X,\mathcal H^1)$ and $H^3_{Zar}(X,\mathcal H^0)$.
By Theorem \ref{cousin_H_et_mu} the latter two groups vanish. Thus, there is a short exact sequence
$0\to H^1_{Zar}(X,\mathcal H^2)\to H^{3}_{et}(X,\mu^{\otimes 2}_n)\to H^0_{Zar}(X,\mathcal H^3)$.
By Theorem \ref{cousin_H_et_mu} the map
$H^0_{Zar}(X,\mathcal H^3)\xrightarrow{\eta^*} H^0_{Zar}(\eta,\mathcal H^3)=H^3_{et}(\eta,\mu^{\otimes 2}_n)$
is injective. Thus, the sequence
$0\to H^1_{Zar}(X,\mathcal H^2)\to H^{3}_{et}(X,\mu^{\otimes 2}_n)\xrightarrow{\eta^*} H^3_{et}(\eta,\mu^{\otimes 2}_n)$
is short exact. This proves the equality
$H^1_{Zar}(X,\mathcal H^2)=NH^{3}_{et}(X,\mu^{\otimes 2}_n)$.
Joining the latter equality with Lemma \ref{prop:Suslin_exact}
we get the following mixed characteristic version of the Suslin exact sequence
The following sequence is short exact
\begin{equation}
\label{Suslin_sequence_2}
0\to H^1_{Zar}(X,\underline K_2)/n \xrightarrow{\alpha} NH^{3}_{et}(X,\mu^{\otimes 2}_n) \xrightarrow{\beta} {_{n}}H^2_{Zar}(X,\underline K_2)\to 0,
\end{equation}
Theorem \ref{thm:Suslin_exact} is proved.
\end{proof}

Let $A$ be a {\it henzelian} d.v.r. Suppose the closed point $v$ of $V=Spec(A)$ is such that its residue field is finite.
Let $l$ be a prime different of $p$. For an abelian group $B$ write $_{\{l\}}B$ for the $l$-primary torsion subgroup of $B$.


\begin{rem}
Since $X$ is smooth projective over $V$ and irreducible and $A$ is henzelian it follows
that the closed fibre $X_v$ is
{\it irreducible too}.
To prove this one has to use the Steiner
decomposition of the morphism $p: X\to V$ as $X\xrightarrow{q} W\xrightarrow{\pi} V$,
where $q$ is surjective with connected fibres and $\pi$ is finite surjective.
Since $q$ is projective and surjective and $X$ is irreducible it follows that
$W$ is irreducible. Since $V$ is local henzelian and $\pi$ is finite it follows
that $W$ has a unique closed point, say $w$. Since the fibres of $q$ are connected
it follows that $X_v=p^{-1}(v)=q^{-1}(w)$ is connected. Since $X_v$ is $v$-smooth,
hence $X_v$ is irreducible.
\end{rem}

\begin{proof}[Proof of Theorem \ref{NH3=torsion_CH}]
Consider a commutative diagram of the form
\begin{equation}
\label{Big_Diagram_2}
    \xymatrix{
0 \ar[r]^{} & H^1_{Zar}(X,\underline K_2)\otimes \mathbb Q_l/\mathbb Z_l \ar[r]^{\alpha} \ar[d]_{\phi} & NH^{3}_{et}(X,\mathbb Q_l/\mathbb Z_l(2)) \ar[r]^{\beta} \ar[d]_{\psi} & {_{\{l\}}}H^2_{Zar}(X,\underline K_2) \ar[r]^{} \ar[d]_{\rho}
& 0 \\
0 \ar[r]^{} & H^1_{Zar}(X_v,\underline K_2)\otimes \mathbb Q_l/\mathbb Z_l \ar[r]^{\alpha_v}  & NH^{3}_{et}(X_v,\mathbb Q_l/\mathbb Z_l(2)) \ar[r]^{\beta_v}  & {_{\{l\}}}H^2_{Zar}(X_v,\underline K_2)  \ar[r]^{}
& 0\\
}
\end{equation}
The upper sequence is exact by Theorem \ref{thm:Suslin_exact}. The bottom sequence is the original Suslin exact sequence.
It is known \cite[Theorem 2.11]{Pa} that $H^1_{Zar}(X_v,\underline K_2)\otimes \mathbb Q_l/\mathbb Z_l=0$ and $\beta_v$ is an isomorphism.
The map $\psi$ is injective by the proper base change theorem. Thus,
$H^1_{Zar}(X,\underline K_2)\otimes \mathbb Q_l/\mathbb Z_l=0$
and the map $\beta$ is an isomorphism. These yield the injectivity of $\rho$.
One can show that the map $\rho$ is surjective. Hence $\rho$ is an isomorphism.
So, the map $\psi$ is an isomorphism too.
The assertion (d) will be proved below in the proof of Theorem
\ref{NH3=torsion_CH_2}.
The theorem is proved.
\end{proof}

\begin{proof}[Proof of Corollary \ref{thm:Roitman}]
Since $\bar A$ is algebraically closed and $d=2$ it follows that
for each integers $s$ and $n$ one has $H^3(\bar A(X_v), \mu^{\otimes s}_{l^n})=0$.
Thus,
$NH^{3}_{et}(X_v,\mathbb Q_l/\mathbb Z_l(2))=H^{3}_{et}(X_v,\mathbb Q_l/\mathbb Z_l(2))$.
The proper base change theorem shows that
$$H^{3}_{et}(X,\mathbb Q_l/\mathbb Z_l(2))=H^{3}_{et}(X_v,\mathbb Q_l/\mathbb Z_l(2)).$$
The item (c) of Theorem \ref{NH3=torsion_CH} shows now that
$NH^{3}_{et}(X,\mathbb Q_l/\mathbb Z_l(2))=H^{3}_{et}(X,\mathbb Q_l/\mathbb Z_l(2))$.
Hence the map $\beta$ gives rise to an isomorphism
$\beta: H^{3}_{et}(X,\mathbb Q_l/\mathbb Z(2))\to {_{\{l\}}}H^2_{Zar}(X,\underline K_2)$.
The second assertion of the corollary is proved. To prove the first one consider
the commutative diagram
\begin{equation}
\label{Big_Diagram_3}
    \xymatrix{
{_{\{l\}}}H^2_{Zar}(X,\underline K_2) \ar[r]^{\beta} \ar[d]^{\rho} & {_{\{l\}}}Alb(X/V) \ar[d]^{\epsilon}  \\
{_{\{l\}}}H^2_{Zar}(X_v,\underline K_2) \ar[r]^{\beta_v}  & {_{\{l\}}}Alb(X_v).  \\
}
\end{equation}
The map $\rho$ is an isomorphism by Theorem \ref{NH3=torsion_CH}. The map $\epsilon$ is an isomorphism
since $A$ is henzelian and $l$ is a prime different of $p$. The map
$\beta_v$ is an isomorphism by the Roitman theorem (or due to \cite[Thm. 2.11]{Pa} and the equality
$NH^{3}_{et}(X_v,\mathbb Q_l/\mathbb Z_l(2))=H^{3}_{et}(X_v,\mathbb Q_l/\mathbb Z_l(2))$
proven just above). Thus, $\beta$ is an isomorphism.

\end{proof}


\begin{proof}[Proof of Theorem \ref{NH3=torsion_CH_2}]
To prove assertions (a), (b) and (c)
repeat part of the arguments from the proof of Theorem \ref{NH3=torsion_CH}.

To prove the assertion (d) it is sufficient to prove that the group
$H^{3}_{et}(X,\mathbb Q_l/\mathbb Z(2))=H^{3}_{et}(X_v,\mathbb Q_l/\mathbb Z(2))$
is finite.
By the Weil conjecture \cite{D74} the group $H^{3}_{et}(X_v,\mathbb Q_l(2))$ vanishes.
It follows that the boundary map
$\delta: H^{3}_{et}(X_v,\mathbb Q_l/\mathbb Z(2))\to H^{4}_{et}(X_v,\mathbb Z(2))$
in the coefficient exact cohomology sequence identifies
$H^{3}_{et}(X_v,\mathbb Q_l/\mathbb Z(2))$
with
${_{\{l\}}}H^{4}_{et}(X_v,\mathbb Z(2))$.
Let $\mathbb F$ be the residue field of $A$ and $\Gamma=Gal(\mathbb F^{sep}/\mathbb F)$.
Then there is a short exact sequence of the form
$0\to H^1(\Gamma,H^3_{et}(\bar X_v,\mathbb Z_l(2)))\to H^{4}_{et}(X_v,\mathbb Z(2))\to H^0(\Gamma,H^4_{et}(\bar X_v,\mathbb Z_l(2)))\to 0.$

By \cite{D} for each integer $i$ the $\mathbb Z_l$-module $H^i_{et}(\bar X_v,\mathbb Z_l(2))$ is finitely generated.
Thus, the $\mathbb Z_l$-module $H^{4}_{et}(X_v,\mathbb Z(2))$ is also finitely generated.
Hence its subgroup
${_{\{l\}}}H^{4}_{et}(X_v,\mathbb Z(2))$
is finite and so is the group
$H^{3}_{et}(X_v,\mathbb Q_l/\mathbb Z(2))$.
The theorem is proved.
\end{proof}

\end{document}